\documentclass{article}
\usepackage{bm}
\usepackage{tikz,mathpazo}
\usepackage{pgf}
\usepackage[top=2.5cm, bottom=2.5cm, left=3cm, right=3cm]{geometry}   
\usepackage{indentfirst}
\usepackage{graphicx}
\usepackage{graphics}
\usepackage[toc,page,title,titletoc,header]{appendix}
\usepackage{bm}
\usepackage{bbm}
\usepackage{listings}  
\usepackage{amsmath}
\usepackage{setspace} 
\usepackage{indentfirst}
\usepackage{caption}
\usepackage{multirow} 
\usepackage{lipsum,multicol}
\usepackage{pdfpages}
\usepackage{float}
\usepackage{amsthm}
\usepackage{pdfpages}
\usepackage{url}
\usepackage{colortbl}
\usepackage{subfigure}
\usepackage{epsfig}
\usepackage{epstopdf}
\usetikzlibrary{shapes.geometric, arrows}
\usepackage{fancyhdr}
\usepackage{abstract}
\usepackage{tikz,mathpazo}
\usetikzlibrary{shapes.geometric, arrows}
\usepackage{amssymb}
\usepackage{latexsym}
\usepackage{verbatim}
\usepackage[numbers]{natbib} 
\usepackage{booktabs}
\allowdisplaybreaks[4]

\newcommand{\inner}[1]{\left\langle #1 \right\rangle}
\newcommand{\norm}[1]{\left\Vert #1\right\Vert}
\newcommand{\bb}[1]{\mathbb{#1}}

\newcommand{\X}{{ \ca{X} }}

\newcommand{\ca}[1]{\mathcal{#1}}

\newcommand{\Diag}[0]{\mathrm{Diag}}

\newcommand{\M}[0]{\mathcal{M}}

\newcommand{\tp}{^\top}

\newcommand{\A}{\ca{A}}

\newcommand{\xk}{{x_{k} }}
\newcommand{\yk}{{y_{k} }}

\newcommand{\wk}{{w_{k} }}

\newcommand{\xkp}{{x_{k+1} }}

\newcommand{\Mxc}{{M_{x,c}}}

\newcommand{\Lxc}{{L_{x,c}}}

\newcommand{\sigmaxc}{ \sigma_{x,c} }

\newcommand{\D}{\ca{D}}

\newcommand{\TX}{{\ca{T}_{\X}}}
\newcommand{\NX}{\ca{N}_{\X}}

\newcommand{\Rn}{\mathbb{R}^n}
\newcommand{\Rp}{\mathbb{R}^p}

\newcommand{\Lsc}{ \sigma_{x,c} }

\newcommand{\Lc}{ {L_{x,c}} }

\newcommand{\Omegax}[1]{ {\Omega_{#1}} }

\newcommand{\K}{\ca{K}}

\def\Thetax{{\bb{B}(x;\,\gamma_x)}}

\usepackage{hyperref}
\hypersetup{
    colorlinks=true,
    linkcolor=blue,
    filecolor=magenta,      
    urlcolor=blue,
    citecolor=blue,
    }

\newtheorem{theo}{Theorem}[section]
\newtheorem{lem}[theo]{Lemma}
\newtheorem{prop}[theo]{Proposition}
\newtheorem{examples}[theo]{Example}
\newtheorem{cond}[theo]{Condition}
\newtheorem{coro}[theo]{Corollary}

\newtheorem{rmk}[theo]{Remark}
\newtheorem{assumpt}[theo]{Assumption}

\usepackage[figuresright]{rotating}
\usepackage{pdflscape}

\usepackage{caption}
\usepackage{algorithm}
\usepackage{algpseudocode}
\usepackage{longtable}
\usepackage{appendix}
\usepackage{tikz-cd}

\numberwithin{equation}{section}

\usepackage{array}

\title{A Quadratically Convergent Alternating Projection Method for Nonconvex Sets}
\author{
Nachuan Xiao\thanks{School of Data Science, The Chinese University of Hong Kong, Shenzhen, China. (xncxy@cuhk.edu.cn).},~
Shiwei Wang\thanks{
Institute of Applied Mathematics, Academy of Mathematics and Systems Science, Chinese Academy
of Sciences, Beijing, China. This work is done while the author is a visiting scholar at the 
Institute of Operational Research and Analytics, National University of Singapore, Singapore. (wangshiwei@amss.ac.cn).},~
Tianyun Tang\thanks{Department of Statistics, University of Chicago (ttang@nus.edu.sg).},~ 
Kim-Chuan Toh\thanks{Department of Mathematics, and Institute of Operations Research and Analytics, National University of Singapore, Singapore 119076. (mattohkc@nus.edu.sg).}}

\begin{document}
\maketitle

\begin{abstract}
    In this paper, we consider the feasibility problem, which aims to find a feasible point for the constraint set $\{x \in \Rn: c(x) = 0\}$ over a possibly non-regular subset $\X \subset \Rn$. Under the constraint nondegeneracy condition, we propose a modified alternating projection method. In our proposed method, based on the concept of projective mapping for $\X$, we alternate a Newton step for finding an inexact solution within the limiting tangent cone of $\X$ and a projection to $\X$. Under mild conditions, we prove the local quadratic convergence of our proposed method. 
    % Additionally, we develop a quadratically convergent Bregman proximal method based on our proposed alternating projection method. 
    Preliminary numerical experiments demonstrate the high efficiency of our proposed alternating projection method. 
\end{abstract}

\section{Introduction}

In this paper, we consider the following feasibility problem,
\begin{equation}\label{Prob_NCF}
   \text{Find } x \in \K:=\X  \cap \M,
\end{equation}
where $\X$ and $\M$ are two closed and possibly nonconvex subsets of $\Rn$. Moreover, we make the following assumptions on \eqref{Prob_NCF}.
\begin{assumpt}
    \label{Assumption_X}
    \begin{enumerate}
        \item $\X$ is a closed subset of $\Rn$ with easy-to-compute projection. 
        \item There exists a  locally Lipschitz continuous mapping $c: \Rn \to \Rp$ such that $\M = \{x \in \Rn: c(x) = 0\}$. Moreover, the gradient of $c$, denoted as $\nabla c(x) = [\nabla c_1(x), \ldots, \nabla c_p(x)]$, is locally Lipschitz continuous over $\Rn$. 
    \end{enumerate}
\end{assumpt}
Here, for any $x \in \Rn$, its projection onto $\X$ is denoted as $\Pi_{\X}(x) := \mathop{\arg\min}_{y \in \X} ~ \norm{x-y}$. It is worth mentioning that $\X$ is possibly non-regular \cite[Definition 2.4.6]{clarke1990optimization}, and thus possibly nonconvex. Moreover, as the mapping $c$ is possibly nonlinear, the $\M$ is also possibly a nonconvex subset of $\Rn$. 

The feasibility problem in the form of \eqref{Prob_NCF} has wide applications in various areas, such as engineering \cite{byrne1993iterative}, computer science \cite{Barrett2018}, and applied mathematics \cite{cai2010singular}. In particular, problem \eqref{Prob_NCF} plays a fundamental role in constrained optimization with the feasible region $\{x \in \X : c(x) = 0\}$. As shown in \cite{xiao2025ongoing}, such constraints are sufficiently general to encompass a wide range of constrained optimization problems like nonlinear programming \cite{bertsekas1997nonlinear}, second order cone programming \cite{alizadeh2003second}, semidefinite programming \cite{toh2008inexact}, and so on \cite{tang2024feasible, tang2024solving}. Since the projection onto $\K$ is generally intractable, the iterates produced by most existing algorithms (e.g.,  augmented Lagrangian methods \cite{zhao2010newton, wen2010alternating}, sequential quadratic programming \cite{mohammadi2020superlinear, izmailov2012stabilized}, and interior-point methods \cite{nemirovski2008interior, zhou2004polynomiality}) may not be feasible. When these algorithms terminate with moderate accuracy, the resulting solution $\hat{z}$ typically only has moderate accuracy in its feasibility. Consequently, to obtain a solution with high accuracy in its feasibility, it is necessary to find a feasible point $z^*$ in the neighborhood of $\hat{z}$, which is equivalent to solving the feasibility problem \eqref{Prob_NCF}.

A conceptually simple and widely used approach for solving feasibility problems of the form \eqref{Prob_NCF} is the alternating projection method. At each iteration, this method first projects the current iterate onto $\M$, then projects the resulting point onto $\X$. When both $\X$ and $\M$ are closed and convex, the alternating projection method converges at a sublinear rate \cite{bregman1965method}, and linear convergence is achieved when the intersection of the relative interiors of $\M$ and $\X$ is nonempty \cite{bauschke1993convergence}.

For the nonconvex feasibility problem (i.e., at least one of $\X$ or $\M$ is nonconvex), the authors of \cite{lewis2008alternating, lewis2009local} establish the first local linear convergence guarantees for the alternating projection method under certain metric regularity conditions. Subsequently, \cite{bauschke2013restrictedB, bauschke2013restricted} introduce the concept of restricted normal cones and further establish local linear convergence for the 
alternating projection method in nonconvex settings. Additionally, \cite{drusvyatskiy2015transversality} proves the local linear convergence of the alternating projection method under transversality conditions. However, all these results provide only local linear convergence rates for general nonconvex feasibility problems. To achieve quadratic convergence, \cite{schost2016quadratically} proposes a modified alternating projection method for the feasibility problem on the intersection
between an affine space and a manifold $\X$ with $\Pi_{\X}$ being locally $C^2$ under transversality conditions. To the best of our knowledge, no existing work provides a superlinearly convergent alternating projection method for general nonconvex feasibility problems. 

On the other hand, to achieve local superlinear convergence for problem \eqref{Prob_NCF}, another line of research is developed based on the framework of the Levenberg–Marquardt (LM) method \cite{fischer2024levenberg, kanzow2005withdrawn,dan2002convergence}. In these LM-based methods, the feasibility problem \eqref{Prob_NCF} is reformulated as the following constrained optimization problem,
\begin{equation}
    \label{Prob_Con}
    \min_{x \in \X} \quad \frac{1}{2} \norm{c(x)}^2.
\end{equation}
At the $k$-th iteration, these LM-based methods select a surrogate function $q_k(x)$ to approximate $\frac{1}{2} \norm{c(x)}^2$ around $\xk$, and compute $\xkp$ by approximately solving the subproblem $\min_{x \in \X} ~ q_k(x)$. However, as such a subproblem typically lacks a closed-form solution even if $q_k$ is a quadratic function, its solution requires iterative subsolvers (e.g., projected gradient methods \cite{nesterov1983method,nesterov2013gradient} or semi-smooth Newton methods \cite{ito2009semi}). This multi-loop structure significantly undermines the computational efficiency of these LM-based methods.

Therefore, existing alternating projection methods for general nonconvex feasibility problems achieve only local linear convergence, while LM-based methods involve solving complicated subproblems. This motivates the following question:
\begin{quote}
    Can we develop an efficient alternating projection method that achieves local superlinear/quadratic convergence for nonconvex feasibility problems?
\end{quote}

\subsection{Motivation}
To find feasible points in the intersection of a low-rank variety and an affine space, \cite{schost2016quadratically} proposes a modified alternating projection method named NewtonSLRA. The NewtonSLRA method involves computing the projection onto the intersection of the tangent cone of the low-rank variety and the affine space at each iteration, while the transversality condition guarantees the nonemptyness of such an intersection.

To develop an alternating projection method with local quadratic convergence rate,  at each iterate $\xk$, we consider approximating $\X$ locally around $\xk$ by its Clarke tangent cone $\mathcal{T}_{\X}(\xk)$ \cite[Definition 2.54]{bonnans2013perturbation}. An intermediate iterate $\yk$ is then computed as the projection onto $\mathcal{T}_{\X}(\xk) \cap \M$ rather than directly projecting $\xk$ onto $\M$. 

Though $\mathcal{T}_{\X}(\xk)$ is convex for a general closed set $\X$,  $\mathcal{T}_{\X}(\xk) \cap \M$ usually does not have closed form when $\M$ is not an affine space, which means that projection onto $\mathcal{T}_{\X}(\xk) \cap \M$ can be computationally expensive. To address this challenge, we compute the intermediate iterate $\yk$ as the inexact projection onto $\mathrm{lin}(\TX(\xk)) \cap \M$, where $\mathrm{lin}(\ca{C}) = \ca{C} \cap (-\ca{C})$ for any closed cone $\ca{C}$. Finally, the next iterate $\xkp$ is computed as $\xkp = \Pi_{\X}(\yk)$.

To efficiently compute $\yk$ as an inexact projection onto $\mathrm{lin}(\TX(\xk)) \cap \M$, we first consider the regularity condition for the nonconvex feasibility problem \eqref{Prob_NCF}. We say that a point $x \in \K$ satisfies the nondegeneracy condition \cite[Definition 4.70]{bonnans2013perturbation} if
\begin{equation}
\label{Eq_Cond_Nondeg}
    \nabla c(x)\tp \mathrm{lin}(\ca{T}_{\X}(x)) = \Rp.
\end{equation}

It is worth noting that in \cite{schost2016quadratically}, to obtain the local quadratic convergence of its proposed NewtonSLRA, the intersection transversality of the low rank set ${\cal D}_r:=\{X\in\mathbb{R}^{n\times m}\mid {\rm rank}(X)=r\}$ and the affine space $E$ is required. 
In general, intersection tranversality is weaker than the nondegeneracy condition \eqref{Eq_Cond_Nondeg} we need. Indeed, when $\Phi(u)=(\X-u)\times(\M-u)$ is a convex multifunction and the condition in \cite[Proposition 2.97]{bonnans2013perturbation} holds, we know from \cite[Theorem 2.83 and (2.163)]{bonnans2013perturbation} that intersection transversality is equivalent to the well known Robinson constraint qualification (RCQ) \cite[(2.182)]{bonnans2013perturbation}, which is usually less restrictive 
than the nondegeneracy condition \eqref{Eq_Cond_Nondeg}. 
However, for ${\cal D}_r\cap E$, we can show by direct calculation that transversality is equivalent to the nondegeneracy condition. 
Moreover, the analysis in \cite{schost2016quadratically} depends heavily on the affine property of $\M$ and the property that $\X$ is a smooth manifold with $\Pi_{\X}$ being locally twice continuously  differentiable. In contrast, though we need the nondegeneracy condition \eqref{Eq_Cond_Nondeg} for the well definedness of our algorithm, our proposed algorithm is able to handle \eqref{Prob_NCF} with general non-affine subset $\M$ and possibly non-regular subset $\X$.

Recently, in \cite{xiao2025ongoing}, the successful application of the constraint dissolving approach for solving an equality constrained optimization problem over a convex set by removing the equality constraint through a constraint dissolving mapping has shown the great power of the approach. It is natural to ask how this tool can be used for solving \eqref{Prob_NCF}. 
To answer this question, we need to introduce the concept of {\it projective mapping}, which is 
important for constructing the constraint dissolving mapping. We call $Q:\X \to \bb{R}^{n\times n}$ a projective mapping, if it satisfies the following assumption.  
 \begin{assumpt}
        \label{Assumption_Q}
        \begin{enumerate}
            \item[(1)] $Q$ is continuous over $\X$;
            \item[(2)] $Q(x)$ is positive semi-definite for any $x \in \X$;
            \item[(3)] $\mathrm{null}(Q(x)) = \mathrm{range}(\ca{N}_{\X}(x))$ for all $x \in \X$; 
            \item[(4)] There exists a locally bounded function $\rho: \X \to \bb{R}_+$  such that for any $x \in \X$, it holds that 
        \begin{equation}\label{eq:rhobd}
            \mathrm{dist}(x + Q(x)d, \X)\leq \rho(x)\|d\|^2, \quad \forall ~\norm{d}\leq 1. 
        \end{equation}
        \end{enumerate}
    \end{assumpt}

It is worth noting that Assumption \ref{Assumption_Q} (1)-(3) are frequently applied in the framework of constraint dissolving approach\cite{xiao2023dissolving,xiao2025ongoing}. Assumption\ref{Assumption_Q} (4) is newly added, and we will show its validity for a wide range of closed sets $\X$ later in Section \ref{sec:Q}. 
    
Based on the projective mapping $Q$, we define a
constraint dissolving mapping $\A: \X \to \Rn$  by
    \begin{equation}
        \label{Eq_AQ_general}
        \A(x) := x - Q(x) d(x), \quad d(x) := \nabla c(x) \left(\nabla c(x)^\top Q(x) \nabla c(x) + \tau(\norm{c(x)}) I_p\right)^{-1} c(x).
    \end{equation}
    Here $\tau: [0, +\infty) \to [0, +\infty)$ is a locally Lipschitz continuous and strictly increasing function with $\tau(0) = 0$ and 
    $\tau(1) \leq L_{\tau}$ for a prefixed constant $L_{\tau}> 0$. The regularization term $\tau(\norm{c(x)}) I_p$ is introduced to ensure the well-definedness of $\A(x)$ for any $x \in \X$. Moreover, as proven later, for any $x \in \X$ that satisfies the nondegeneracy condition \eqref{Eq_Cond_Nondeg}, the mapping $\A$ is well-defined in a neighborhood of $x$. 

%%%%%%%%%%%%%%%%%%%%%%%%
    \subsection{Contribution}
    In this paper, we propose a modified alternating projection method for solving the nonconvex feasibility problem \eqref{Prob_NCF}, which employs the following update scheme:
    \begin{equation}
        \label{Eq_Feasibility_RestorationQ}
        \xkp \in \Pi_{\X}(\A(\xk)+ r_k).
    \end{equation}
    Here $r_{k}$ is the evaluation error characterizing the inexactness in the evaluation of $\Pi_{\X}(\A(x_k))$.  Throughout this paper, we assume that 
    \begin{equation}
        \label{Eq_Cond_rk}
        \norm{r_{k}} \leq \xi(\xk) \norm{c(\xk)}^2
    \end{equation}
    holds for any $k\geq 0$, where $\xi: \Rn \to \bb{R}_+$ is a locally bounded function, i.e., $\xi$ is bounded over any compact set. It is worth noting that $\Pi_{\X}(\A(\xk))$ does not necessarily need to be a singleton. Here, the inclusion in \eqref{Eq_Feasibility_RestorationQ} means that $x_{k+1}$ can be chosen as an arbitrary point within the subset $\Pi_{\X}(\A(\xk) + r_k)$.

    We prove that, for any $x \in \X \cap \M$ that satisfies the nondegeneracy condition \eqref{Eq_Cond_Nondeg}, there exists a neighborhood $\Omega_x$ of $x$, such that for any $x_0 \in \Omega_x$, the iterates $\{\xk\}$ generated by \eqref{Eq_Feasibility_RestorationQ} converges towards $\X \cap \M$ quadratically. Moreover, by combining \eqref{Eq_Feasibility_RestorationQ} and the projected gradient descent method for \eqref{Prob_Con}, we propose a globally convergent algorithm in Algorithm \ref{Alg:Adap_AP2}. 

    Additionally, when $\X$ is a closed convex subset of $\Rn$ and admits a Bregman kernel function $\phi$, we propose the following Bregman proximal-based method
    \begin{equation}
        \label{Eq_BLM}
        \xkp = (\nabla \phi)^{-1}\left( \nabla \phi(\xk) - d(\xk)  \right).
    \end{equation}
    We prove the local quadratic convergence of \eqref{Eq_BLM} by showing that it fits into the update scheme \eqref{Eq_Feasibility_RestorationQ}. 

    Furthermore, we perform preliminary numerical experiments to evaluate the performance of Algorithm \ref{Alg:Adap_AP2} on  feasibility problems that admit possibly nonconvex $\X$, including the low-rank variety and the $\ell_{q}$-norm ball with $q \in (0,1)$. Our proposed Algorithm \ref{Alg:Adap_AP2} demonstrates its fast local quadratic convergence in all the test instances. In particular, on the feasibility problem of a low-rank variety intersecting an affine space, our proposed Algorithm \ref{Alg:Adap_AP2} shows significant superiority over the NewtonSLRA method  in \cite{schost2016quadratically}. These numerical experiments further demonstrate the promising potential of our proposed Algorithm \ref{Alg:Adap_AP2}.

    \subsection{Organization}
    The outline of the rest of this paper is as follows. In Section 2, we present the notations and preliminary concepts that are necessary for the proofs in this paper. In Section 3, we prove the local quadratic convergence of the scheme \eqref{Eq_Feasibility_RestorationQ}, and present the convergence properties of Algorithm \ref{Alg:Adap_AP2}. In Section 4, we specialize  \eqref{Eq_Feasibility_RestorationQ} to a 
    Bregman proximal-based method for solving \eqref{Prob_NCF} when $\X$ is a closed convex set that admits a Bregman kernal function. 
    Preliminary numerical experiments are presented in Section 5 to demonstrate the efficiency of our proposed algorithm. We conclude the paper in the last section.

%%%%%%%%%%%%%%%%%%%%%%%%%%
\section{Preliminaries}
In this section, we introduce some notation and basic concepts that 
are frequently used in this paper.

For any matrix $A\in\bb{R}^{n\times p}$, 
let $\mathrm{range}(A)$ be the subspace spanned by the column vectors of $A$, $\mathrm{null}(A)$ be the null space of $A$ (i.e., $\mathrm{null}(A) = \{d \in \Rp: A d = 0\}$), 
and $\norm{\cdot}$ denotes the $\ell_2$-norm of a vector or an operator. 
For a subset $\ca{C} \subseteq \Rn$, $\mathrm{range}(\ca{C})$ refers to the smallest subspace of $\Rn$ that contains $\ca{C}$, $\mathrm{lin}(\ca{C})$ refers to the largest subspace of $\Rn$ that is contained in $\ca{C}$, $\mathrm{aff}(\ca{C})$ refers to the affine hull of $\ca{C}$, and $\mathrm{ri}(\ca{C})$ refers to the relative interior of $\ca{C}$. Moreover, when $\ca{C}$ is a subspace of $\Rn$, $\ca{C}^{\perp}$ is defined as the largest subspace that is orthogonal to $\ca{C}$. Additionally, for any $w \in \Rn$, we use the notation $\inner{w,\ca{C}} := \{ \inner{w,d} : d\in \ca{C}\}$ and $w^{\perp} := \{w\}^{\perp}$.
We use $\bb{B}(x;r) = \{ y\in \bb{R}^n: \norm{y-x}\leq r\}$ to denote the closed ball with center  $x$ and radius $r$.

The notation $\mathrm{diag}(A)$ and $\Diag(x)$
stand for the vector formed by the diagonal entries of a matrix $A$,
and the diagonal matrix with the entries of $x\in\bb{R}^n$ as its diagonal, respectively. 
We denote the $r$-th largest singular value of a matrix $A\in \bb{R}^{n\times p}$ by $\sigma_r(A)$, while $\sigma_{\min}(A)$ refers to the smallest singular value of 
$A$. Furthermore, 
the pseudo-inverse of $A$ is denoted by $A^\dagger \in \bb{R}^{p\times n}$, which satisfies $AA^\dagger A = A$, $A^\dagger AA^\dagger = A^\dagger$, and both $A^{\dagger} A$ and $A A^{\dagger}$ are symmetric \cite{golub2013matrix}.

For any closed subset $\ca{C} \subseteq \Rn$ and any $x \in \bb{R}^n$, we define the projection from $x \in \bb{R}^n$ to $\ca{C}$ as 
\begin{equation*}
	\Pi_{\ca{C}}(x) := \mathop{\arg\min}_{y \in \ca{C}} ~ \norm{x-y}. 
\end{equation*}  
Furthermore, $\mathrm{dist}(x, \ca{C})$ refers to the distance between $x$ and $\ca{C}$, i.e. $ \mathrm{dist}(x, \ca{C}) = \norm{x - \Pi_{\ca{C}}(x)}$.

% Most of them are adopted directly from \cite{bonnans2013perturbation,rockafellar2009variational}. 
For a set-valued mapping $\mathcal{F}:\mathbb{R}^n\rightrightarrows\mathbb{R}^n$, 
we define
\begin{equation*}
\begin{aligned}
    & \liminf\limits_{{ x'}\rightarrow x}\mathcal{F}(x'):=\left\{d\in\mathbb{R}^n :\limsup_{x'\rightarrow x}\; ({\rm dist}(d, \mathcal{F}(x')))=0
    \right\},\\
    & \limsup\limits_{{ x'}\rightarrow x}\mathcal{F}(x'):=\left\{d\in\mathbb{R}^n :\liminf_{x'\rightarrow x}\; ({\rm dist}(d, \mathcal{F}(x')))=0
    \right\}. 
\end{aligned}
\end{equation*}

Suppose ${\cal X}$ is any nonempty subset in $\mathbb{R}^n$. The Clarke tangent cone \cite[Definition 2.54]{bonnans2013perturbation} of ${\cal X}$ at $x$ is a closed convex cone defined by
\begin{equation}\label{eq:defT}{\mathcal{T}}_{\X}(x):= \liminf_{t\downarrow0,\, x'\overset{\X}{\rightarrow} x}\frac{\X-x'}{t}.
\end{equation}
Other frequently used tangent cones like the contingent (Bouligand) cone and the inner tangent cone are also defined in \cite[Definition 2.54]{bonnans2013perturbation}. Their relationships with the Clarke tangent cone are given in \cite[page 45]{bonnans2013perturbation}. It is also known from \cite[Proposition 2.55]{bonnans2013perturbation} that they are the same when $\X$ is convex. For the concrete examples considered in Section \ref{sec:Q}, they are different only for $l_0$-norm. 
The following definitions of the regular and limiting normal cones of sets are taken from \cite[Page 213]{rockafellar2009variational}, respectively. 
	Let  $x\in {\cal X}$ be given. 
	We call $\widehat{\cal N}_{\cal X} (x):= \{d\in\mathbb{R}^n\mid \langle d, x'-x\rangle = o(\|x'-x\|) \quad \forall\, x\in {\cal X}\}$ as 
	the regular normal cone to  ${\cal X}$ at point $x$ and
	\begin{equation*}
		{\cal N}_{\cal X} (x):= \limsup\limits_{x'\rightarrow x}\widehat{\cal N}_{\cal X} (x)
	\end{equation*}
	as the limiting normal cone  (also known
	as  Mordukhovich normal cone or basic normal cone) to ${\cal X}$ at point $x$. The polarity relationship of the above sets are discussed in \cite[Figure 6-17]{rockafellar2009variational}.

\section{Convergence analysis}\label{seca}
In this section, we present the convergence analysis of our proposed scheme \eqref{Eq_Feasibility_RestorationQ}. Section \ref{Subsection_Constants} presents the basic definitions of the constants that are essential in our theoretical analysis. We establish the local quadratic convergence properties for \eqref{Eq_Feasibility_RestorationQ} in Section \ref{Subsection_localquadrate}. In Section \ref{Subsection_global_alg}, we develop a globally convergent alternating projection method by hybridizing \eqref{Eq_Feasibility_RestorationQ} with the projected gradient descent method for solving \eqref{Prob_Con}. 

\subsection{Constants and basic properties}
\label{Subsection_Constants}
In this subsection, we present the definitions of some basic constants that are required in our theoretical analysis. 
We begin with the following lemma illustrating the positive-definiteness of $\nabla c(x)\tp Q(x) \nabla c(x)$ for any $x \in \K$. 
 \begin{lem}
    \label{Le_aux_0}
     Suppose Assumption \ref{Assumption_X} and Assumption \ref{Assumption_Q} hold. Then for any $x \in \K$ that satisfies the nondegeneracy condition, it holds that
     \begin{equation}
         \sigma_{\min}(\nabla c(x)\tp Q(x) \nabla c(x)) > 0. 
     \end{equation}
 \end{lem}
\begin{proof}
For any $x \in \K$ that satisfies the nondegeneracy condition,  $\nabla c(x)\tp Q(x) \nabla c(x)$ is positive semi-definite. To prove the definiteness of $\nabla c(x)\tp Q(x) \nabla c(x)$, we first assume that there exists a nonzero $d \in \Rp$ such that $\nabla c(x)\tp Q(x) \nabla c(x)d = 0$. Then we can conclude that $\nabla c(x) d \in \mathrm{null}(Q(x))=\mathrm{range}(\ca{N}_\X(x))$. Thus $\mathrm{range}(\nabla c(x)) \cap \mathrm{range}(\NX(x)) \neq \{0\}$. We also know that $\mathrm{range}(\ca{N}_\X(x))$ $=(\mathrm{lin}(\TX(x)))^\perp$, whose proof is similar to that of \cite[(4.174)]{bonnans2013perturbation} with a slight modification by using \cite[Theorem 6.28 (b)]{rockafellar2009variational}.  
Combining the above discussion together, we know that this contradicts the nondegeneracy condition in Definition \ref{Eq_Cond_Nondeg}. Therefore, we can conclude that $\nabla c(x)\tp Q(x) \nabla c(x) \succ 0$. This completes the proof. 
\end{proof}

Therefore, for any $x \in \K$ that satisfies the nondegeneracy condition \eqref{Eq_Cond_Nondeg}, we define
\begin{equation}
    \label{Eq_Defin_sigmas}
    \sigma_{x, Q} := \sigma_{\min}(\nabla c(x)\tp Q(x) \nabla c(x)), \quad \sigmaxc := \sigma_{\min}(\nabla c(x)),
\end{equation}
where $\sigma_{\min}(\cdot)$ is the smallest singular value. 
Since the nondegeneracy condition \eqref{Eq_Cond_Nondeg} is satisfied at $x \in \K$, we have $\sigmaxc > 0$ from \cite[subsection 4.6.]{bonnans2013perturbation}.

Next, for $x\in\K$, we define $\gamma_x$ as 
\begin{equation}\label{eq-gamma-x}
    \gamma_x := \mathop{\arg\max}_{0 < \gamma \leq 1} \Big\{ \gamma \;
    :\;  \inf_{y \in\bb{B}(x;\gamma) } \sigma_{\min}(\nabla c(y)\tp Q(y) \nabla c(y))  \geq \frac{1}{2} \sigma_{x, Q}\Big\}. 
\end{equation}
Note that from Lemma \ref{Le_aux_0}, $\sigma_{x,Q}>0$.
Here the subscript of $\gamma_x$ emphasizes its dependence on 
$x$. Based on the definition of $\gamma_x$, we define 
several constants in Table \ref{Table_Constants}. 
 We also define the neighborhood $\Omega_x := \X\cap \bb{B}(x;\delta_x)$. 
\begin{table}[tb]
\centering
\begin{tabular}{l|l}
\hline
\textbf{Constants} & \textbf{Definition}  \\ \hline
% $\sigma_{x, Q}$ & $\sigma_{\min}(\nabla c(x)\tp Q(x) \nabla c(x))$  \\ \hline
%$\Thetax$ & $\bb{B}(x;\gamma_x)  $ 
%\\ \hline
$\Mxc$ & $\sup_{y \in \Thetax}   \norm{\nabla c(y)}$  \\ \hline
$M_{x,Q}$ & $\sup_{y \in \Thetax}  \norm{Q(y)}$  \\ \hline
$M_{x, \xi}$ & $\sup_{y \in \Thetax}  \xi(y)$  \\ \hline
$\Lxc$ & $\sup_{y, z \in \Thetax, ~ y\neq z}   \frac{\norm{\nabla c(y) - \nabla c(z)}}{\norm{y-z}}$  \\ \hline
$L_{x, \rho}$ & $\sup_{y\in \Thetax}   \rho(y)$  \\ \hline
% $\delta_{c,Q}$ & $$ \\ \hline
$\delta_x$ & 
$\min\left\{
\begin{array}{l}
\frac{\gamma_x}{2},\, \frac{1}{\Mxc},  \,
\, \frac{\sigma_{x,c}}{6L_{x,c}},\,
\frac{\sigma_{x,Q}}{2 M_{x,c}^2},\, 
\frac{\gamma_x\sigma_{x, Q}}{\sigma_{x, Q}+ 4M_{x, Q}M_{x, c}^2}, \,
\frac{M_{x,Q}\sigma_{x,Q}}{4L_{x,\rho}M_{x,c}^2+M_{x,\xi}\sigma_{x,Q}^2},
\\
\frac{\sigma_{x,Q}^2}{(12L_{x, \rho}M_{x, c}^3 + 2 L_{\tau}\sigma_{x, Q} 
+ 2L_{x, c} M_{x, Q}^2 M_{x, c}^2+ 2\sigma_{x, Q}^2  M_{x,c} M_{x, \xi}) 2M_{x,c}}  
\end{array}\right\}$ 
\\[20pt] \hline
$\Omega_x$ & $\X\cap \bb{B}(x;\delta_x)  \subset \Thetax$ 
\\ \hline
$\Xi_x$ & $\X\cap \bb{B}(x;\delta_x') \subset \Omega_x$, where
$ \delta_x' =  \frac{ \sigma_{x, Q}}{32 M_{x, Q} M_{x, c}^2 + 2\sigma_{x, Q}}\delta_x$ 
 \\ \hline
\end{tabular}
\caption{Definitions of constants associated with $\Thetax$. 
}
\label{Table_Constants}
\end{table}

In the following lemma, we illustrate the relationship between $\|c(y)\|$ and $\mathrm{dist}(y, \M)$ for any  $x \in \K$ and any $y \in \Omega_x$. 
It follows from Lemma~\ref{Le_aux_relationship_c_dist_PEJcc} that for any $x \in \K$, $\norm{c(y)} \leq 1$ for all $y \in \Omega_x$.
\begin{lem}
    \label{Le_aux_relationship_c_dist_PEJcc}
    Suppose Assumption \ref{Assumption_X}  holds. Then for any $x \in \K$ that satisfies the nondegeneracy condition, it holds for any 
    $y \in \Omegax{x}$ that
    \begin{equation*}
        \frac{1}{M_{x,c}} \norm{c(y)} \leq \mathrm{dist}(y, \M) \leq \frac{2}{\sigma_{x, c} } \norm{c(y)}.
    \end{equation*}
\end{lem}
\begin{proof}
    For any $x \in \K$ and any $y \in \Omega_x$, it follows from a similar proof of  \cite[Lemma 1]{xiao2025cdopt} and the convexity of $\Thetax$ that 
    $\norm{c(y)} \leq M_{x,c} \mathrm{dist}(y, \M)$, and hence
    \begin{equation*}
        \frac{1}{M_{x,c}}\norm{c(y)} \leq  \mathrm{dist}(y, \M).
    \end{equation*}

    Let $z \in \Pi_{\M}(y)$, then from the definition of $z$, we can conclude that $\norm{z-y}\leq \norm{y-x}$ since $x\in\M$. 
    Thus $\norm{z-x} \leq \norm{z-y} + \norm{y-x} \leq 2\delta_x \leq \gamma_x$, and hence $z \in \Thetax$. By the mean-value theorem, for any fixed $\nu \in \bb{R}^p$, there exists a point $\xi_\nu \in  \bb{R}^n$ that is a convex combination of $y$ and $z$ such that $ \nu\tp c(y) = (y-z)\tp\nabla c(\xi_\nu)\nu $. 
    Then it follows from the definition of $z$ that $y-z \in \mathrm{range}(\nabla c(z))$. Moreover, using the definition of $L_{x,c}$, 
    we have that
    \begin{equation}
    \label{Eq_Le_aux_relationship_c_dist_PEJcc_0}
        \sigma_{\min}(\nabla c(z)) \geq \sigma_{\min}(\nabla c(x)) - L_{x, c} \norm{z - x} \geq \sigma_{x, c}- 2\delta_x L_{x, c}. 
    \end{equation}

    As a result, let $\tilde{\nu} = \frac{\nabla c(z)\tp(y-z)}{\norm{\nabla c(z)\tp(y-z)}}$, we have
	\begin{equation*}
		\begin{aligned}
			&\norm{c(y)}= \sup_{\nu \in \bb{R}^p, \norm{\nu} = 1} (y-z)\tp\nabla c(\xi_\nu)\nu  \geq (y-z)\tp\nabla c(\xi_{\tilde{\nu}})\tilde{\nu}  \\
			={}&  (y-z)\tp\nabla c(z)\tilde{\nu} - (y-z)\tp\left(\nabla c(z) - \nabla c(\xi_{\tilde{\nu}})\right)\tilde{\nu}  \\
			={}& \norm{ \nabla c(z)\tp (y-z)} - (y-z)\tp\left(\nabla c(z) - \nabla c(\xi_{\tilde{\nu}})\right)\tilde{\nu} \\
			\geq{}& \norm{ \nabla c(z)\tp (y-z)} - \Lc \norm{y-z}^2 
            \\
			\geq{}& (\Lsc - 3\delta_x \Lc )\mathrm{dist}(y, \M)  \geq  \frac{\Lsc }{2} \mathrm{dist}(y, \M).
		\end{aligned}
	\end{equation*}
    Here the second inequality uses the fact that $y-z \in \mathrm{range}(\nabla c(z))$ and $\sigma_{\min}(\nabla c(z)) \geq \sigma_{x, c} - 2\delta_x L_{x, c}$ from \eqref{Eq_Le_aux_relationship_c_dist_PEJcc_0}.
    This completes the proof. 
\end{proof}

Next, we present the following technical lemma which will be used later.
\begin{lem}
    \label{Le_aux_2}
    Suppose Assumption \ref{Assumption_X} and Assumption \ref{Assumption_Q} hold. Then for any $z, r \in \Rn$, any $w_1 \in \Pi_{\X}(z)$ and any $w_2 \in \Pi_{\X}(z + r)$, it holds that 
    \begin{equation}
        \norm{w_1 - w_2} \leq 2\mathrm{dist}(z, \X) + 2\norm{r}. 
    \end{equation}
\end{lem}
\begin{proof}
    For any $z, r \in \Rn$, any $w_2 \in \Pi_{\X}(z + r)$ and any $w_1 \in \Pi_{\X}(z)$, we have
    \begin{equation}
        \norm{z - w_1} = \mathrm{dist}(z, \X), \quad  \mathrm{dist}(z +r, \X) = \norm{z +r - w_2} \leq \norm{z+r - w_1}.
    \end{equation}
    As a result, we have
    \begin{equation}
        \begin{aligned}
            &\norm{w_1 - w_2} \leq \norm{w_1 - z} + \norm{z - (z+r)} + \norm{(z+r) -w_2} \\
            \leq{}& \norm{w_1 - z} + \norm{z - (z+r)} + \norm{(z+r) -w_1} \\
            \leq{}& 2\norm{w_1 - z} + 2\norm{r} = 2\mathrm{dist}(z, \X) + 2\norm{r}.
        \end{aligned}
    \end{equation}
    This completes the proof. 
\end{proof}

%%%%%%%%%%%%%%%%%%%%%%%%%%%%%%%%%%%%%%%%%%%%%%%%%%
\subsection{Local quadratic convergence}
\label{Subsection_localquadrate}
 In this section, we analyze the local convergence properties of the scheme \eqref{Eq_Feasibility_RestorationQ}. We prove that, when initialized at $x_0$ sufficiently close to $\K$, the sequence ${\xk}$ generated by \eqref{Eq_Feasibility_RestorationQ} converges to $\K$ quadratically.

We begin our theoretical analysis with the following lemma 
bounding the distance between $\A(y)$ and $y$ for any $y \in \Omega_x$. 
 \begin{lem}
    \label{Le_aux_1}
     Suppose Assumption \ref{Assumption_X} and Assumption \ref{Assumption_Q} hold. Then for any $x \in \K$ that satisfies the nondegeneracy condition~\ref{Eq_Cond_Nondeg}, it holds for any $y \in \Omegax{x}$ that 
     \begin{equation}
         \norm{\A(y) - y} \leq \frac{2 M_{x, Q} M_{x, c}}{ \sigma_{x, Q}} \norm{c(y)}. 
     \end{equation}
 \end{lem}
 \begin{proof}
     For any $x \in \K$ that satisfies the nondegeneracy condition and for any $y \in \Omega_x$, it holds that 
     \begin{equation}
         \begin{aligned}
             &\norm{\A(y) - y} = \norm{Q(y) d(y)}
             \leq \frac{\norm{Q(y)} \norm{\nabla c(y)}}{\sigma_{\min}(\nabla c(y)\tp Q(y)\nabla c(y))} \norm{c(y)} \leq \frac{2 M_{x, Q} M_{x, c}}{ \sigma_{x, Q}} \norm{c(y)}. 
         \end{aligned}
     \end{equation}
     Here the first inequality directly follows the definition of $\gamma_x$ and the choices of $\Omega_x$. Moreover, the second inequality uses the definitions in Table \ref{Table_Constants}. 
     This completes the proof. 
 \end{proof}
 
Next, we present the following lemma showing that the distance between $\A(y)$ and $\X$ is proportional to $\norm{c(y)}^2$. 
\begin{lem}\label{lemQ}
    Suppose Assumption \ref{Assumption_X} and Assumption \ref{Assumption_Q} hold. Then for any $x \in \K$ that satisfies the nondegeneracy condition, it holds for any $y \in \Omegax{x}$ that $$\norm{\Pi_{\X}(\A(y)) - \A(y)} \leq \frac{4L_{x, \rho}M_{x, c}^2}{\sigma_{x, Q}^2} \norm{c(y)}^2. $$ 
\end{lem}
\begin{proof}
For any $x \in \K$ that satisfies the nondegeneracy condition, it follows from Lemma \ref{Le_aux_0} and the definition of $\Omega_x$ that $\sigma_{\min}(\nabla c(y)\tp Q(y) \nabla c(y)) \geq \frac{\sigma_{x, Q}}{2}$ holds for all $y \in \Omega_x$. 
Then it holds that 
\begin{equation*}
\norm{d(y)} \leq \norm{\nabla c(y)}\norm{(\nabla c(y)\tp Q(y) \nabla c(y)+ \tau(\norm{c(y)})I_p)^{-1}}
\norm{c(y)} \leq \frac{2M_{x,c}}{\sigma_{x,Q}}\norm{c(y)},
\end{equation*}
where the last inequality uses $\sigma_{\min}(\nabla c(y)\tp Q(y) \nabla c(y)) \geq \frac{\sigma_{x, Q}}{2}$, and $d(y)$ is defined in \eqref{Eq_AQ_general}. 
Moreover, by Lemma~\ref{Le_aux_relationship_c_dist_PEJcc}, 
$\norm{c(y)}\leq M_{x,c} {\rm dist}(y,\M) \leq M_{x,c}\delta_x$, and hence
\begin{equation*}
    \norm{d(y)} \leq  \frac{2M_{x,c}}{\sigma_{x,Q}}M_{x,c}\delta_x \leq 1.
\end{equation*}
Now together with Assumption \ref{Assumption_Q}(4), we have 
    \begin{equation*}
    \norm{\Pi_{\X}(\A(y)) - \A(y)}
            = \mathrm{dist}(y - Q(y) d(y) ,{\cal X}) 
            \leq \rho(y)\norm{d(y)}^2  
            \leq  \frac{4L_{x, \rho}M_{x, c}^2}{\sigma_{x, Q}^2} \norm{c(y)}^2. 
    \end{equation*}
    Here the first inequality follows from the feasibility of $y\in \X$ and Assumption \ref{Assumption_Q}(4). This completes the proof. 
\end{proof}

The following lemma demonstrates that for any $x \in \K$ and any $y \in \Omega_x$, the mapping $\A(y)$ reduces the feasibility violation with respect to $\M$ quadratically. 
\begin{lem}
    \label{lem:step2}
    Suppose Assumption \ref{Assumption_X} and Assumption \ref{Assumption_Q} hold. Then for any $x \in \K$ that satisfies the nondegeneracy condition, it holds for any $y \in \Omegax{x}$ that
    \begin{equation}
        \norm{c(\A(y))} \leq \left(\frac{2 L_{\tau}\sigma_{x, Q}  + 2L_{x, c} M_{x, Q}^2 M_{x, c}^2}{ \sigma_{x, Q}^2} \right) \norm{c(y)}^2. 
    \end{equation}
\end{lem}
\begin{proof}
For any $x \in \K$ that satisfies the nondegeneracy condition, from Lemma \ref{Le_aux_1} and the definitions of $\delta_x$ and $\gamma_x$, we can conclude that 
\begin{equation}
    \begin{aligned}
        &\norm{\A(y) - x} \leq \norm{\A(y)-y} + \delta_x \leq \frac{2M_{x, Q}M_{x, c}}{\sigma_{x, Q}} \norm{c(y)} + \delta_x\\
        \leq{}& \frac{2M_{x, Q}M_{x, c}^2}{\sigma_{x, Q}} \norm{y-x} + \delta_x \leq \left( \frac{\sigma_{x, Q}+ 2M_{x, Q}M_{x, c}^2}{\sigma_{x, Q}} \right) \delta_x \leq \gamma_x. 
    \end{aligned}
\end{equation}
Note that the third inequality uses Lemma \ref{Le_aux_relationship_c_dist_PEJcc}. 
This illustrates that $\A(y)\in\Thetax$. 
Moreover, from the Lipschitz continuity of $c$ over $\Thetax$, it holds that 
    \begin{equation}
        \begin{aligned}
            &\norm{c(\A(y))} \leq  \norm{c(y) + \nabla c(y)\tp (\A(y) - y)} + \frac{L_{x,c}}{2} \norm{\A(y) - y}^2\\
            \leq{}& \norm{c(y) - \nabla c(y)\tp Q(y)  d(y)} 
            + \frac{L_{x,c}}{2} \norm{Q(y)d(y)}^2 \\
            \leq{}& \tau(\norm{c(y)}) \norm{ \big(\nabla c(y)\tp Q(y) \nabla c(y) + \tau(\norm{c(y)}) \big)^{-1} c(y)}  + \frac{L_{x, c}}{2} \norm{Q(y)d(y)}^2\\
            \leq{}& \frac{2L_{\tau}}{\sigma_{x, Q}} \norm{c(y)}^2 + \frac{2L_{x, c} M_{x, Q}^2 M_{x, c}^2}{ \sigma_{x, Q}^2}  \norm{c(y)}^2.
        \end{aligned}
    \end{equation}
    Here the first inequality follows from the mean value theorem, the second inequality uses the expression of $\A$ in \eqref{Eq_AQ_general}, and the last inequality follows from the definition of $\gamma_x$. This completes the proof.  
\end{proof}

Next, the following lemma demonstrates that for any $x \in \K$ and any $y \in \Omega_x$, the mapping $\Pi_{\X}(\A(\cdot))$ reduces the feasibility violation with respect to $\M$ quadratically. 
\begin{lem}
    \label{Le_quadratic_decrease}
    Suppose Assumption \ref{Assumption_X} and Assumption \ref{Assumption_Q} hold. Then for any $x \in \K$ that satisfies the nondegeneracy condition, it holds for any $y \in \Omegax{x}$ that
    \begin{equation}
        \norm{c(\Pi_{\X}(\A(y)))} \leq \left(\frac{4L_{x, \rho}M_{x, c}^3 + 2 L_{\tau}\sigma_{x, Q} + 2L_{x, c} M_{x, Q}^2 M_{x, c}^2}{\sigma_{x, Q}^2} \right) \norm{c(y)}^2.
    \end{equation}
\end{lem}
\begin{proof}
   For any $x \in \K$ that satisfies the nondegeneracy condition and any $y \in \Omegax{x}$, we have that
\begin{equation}
    \begin{aligned}
        &\norm{c(\Pi_{\X}(\A(y)))} \leq \norm{c(\Pi_{\X}(\A(y))) - c(\A(y))} + \norm{c(\A(y))}\\[5pt]
        \leq{}& M_{x, c} \norm{\Pi_{\X}(\A(y)) - \A(y)} + \norm{c(\A(y))}\\
        \leq{}& \frac{4L_{x, \rho}M_{x, c}^3}{\sigma_{x, Q}^2} \norm{c(y)}^2 + \left(\frac{2 L_{\tau}\sigma_{x, Q}  + 2L_{x, c} M_{x, Q}^2 M_{x, c}^2}{ \sigma_{x, Q}^2} \right) \norm{c(y)}^2\\
        ={}& \left(\frac{4L_{x, \rho}M_{x, c}^3 + 2 L_{\tau}\sigma_{x, Q} + 2L_{x, c} M_{x, Q}^2 M_{x, c}^2}{\sigma_{x, Q}^2} \right) \norm{c(y)}^2. 
    \end{aligned}
\end{equation}
Here the second inequality follows from the Lipschitz continuity of $c$ over $\Thetax$, and the third inequality uses Lemma \ref{lemQ} and Lemma \ref{lem:step2}. This completes the proof. 
\end{proof}

Furthermore, the following proposition illustrates that for any $x \in \K$ and any $x_0 \in \Xi_x\subseteq\Omega_x$, the iterates $\{\xk\}$ generated by \eqref{Eq_Feasibility_RestorationQ} are restricted within $\Omega_x$.   
\begin{prop}
    \label{prop:supconv}
    Suppose Assumption \ref{Assumption_X} and Assumption \ref{Assumption_Q} hold. Then for any $x \in \K$ that satisfies the nondegeneracy condition, and any sequence $\{\xk\}$ generated by \eqref{Eq_Feasibility_RestorationQ} with $x_0\in\Xi_x$, it holds that 
    \begin{equation}
        \sum_{i = 0}^k \norm{x_{i+1} - x_i} \leq  \frac{8 M_{x, Q} M_{x, c}^2}{ \sigma_{x, Q}} \mathrm{dist}(x_0, \M).
    \end{equation}
    Additionally, it holds that $\{\xk\} \subseteq \Omega_x$. 
\end{prop}
\begin{proof}
We prove this proposition by induction, that is, we assume that $\{x_i: 0\leq i\leq k\} \subset \Omega_x$. 
Recall that $x_{k+1}\in\Pi_\X(\A(\xk)+r_k)$.
Then from Lemma \ref{Le_quadratic_decrease}, for any $\wk \in \Pi_{\X}(\A(\xk))$, we have 
\begin{equation}
\label{eq-tmp1}
    \begin{aligned}
        &\norm{c(\xkp)} =\norm{c(\xkp)-c(\wk)+c(\wk)}
        \;\leq\; M_{x, c}\norm{\xkp - \wk} + \norm{c(\wk)}
        \\[5pt]
        \leq{}& 2M_{x, c}\mathrm{dist}(\A(\xk), \X) + 2M_{x, c}\norm{r_k} + \norm{c(\wk)}\\
        \leq{}& \frac{8L_{x, \rho} M_{x, c}^3}{\sigma_{x,Q}^2} \norm{c(\xk)}^2 + 2M_{x, c}M_{x, \xi}\norm{c(\xk)}^2 
        %\\
        %&
        + \left(\frac{4L_{x, \rho}M_{x, c}^3 + 2 L_{\tau}\sigma_{x, Q} + 2L_{x, c} M_{x, Q}^2 M_{x, c}^2}{\sigma_{x, Q}^2} \right) \norm{c(\xk)}^2
        \\
        ={}& R_x \norm{c(\xk)}^2 
        \leq \frac{1}{2} \norm{c(\xk)},
    \end{aligned}
\end{equation}
    where $$R_x = \frac{12L_{x, \rho}M_{x, c}^3 + 2 L_{\tau}\sigma_{x, Q} + 2L_{x, c} M_{x, Q}^2 M_{x, c}^2 + 2\sigma_{x, Q}^2  M_{x,c} M_{x, \xi}}{\sigma_{x, Q}^2}.$$
In \eqref{eq-tmp1}, the first inequality uses the Lipschitz continuity of $c$ over $\Thetax$, the second inequality uses Lemma \ref{Le_aux_2} with $z = \A(\xk)$ and $r = r_k$, and the third inequality uses Lemma \ref{lemQ}, \eqref{Eq_Cond_rk}, and Lemma \ref{Le_quadratic_decrease}. Additionally, the last inequality follows from the choice of $\delta_x$. 

It follows from \eqref{eq-tmp1}  that for any $i \leq k$, we have 
\begin{equation*}
    \norm{c(x_{i})} \leq \frac{1}{2^i} \norm{c(x_0)}. 
\end{equation*}
Furthermore, from the update scheme of $\{\xk\}$ in \eqref{Eq_Feasibility_RestorationQ}, it holds for any $i \leq k$ that  
\begin{equation}
    \begin{aligned}
        &\norm{x_{i+1} - x_i} \leq  \norm{\Pi_{\X}(\A(x_i)+ r_i) - \A(x_i) } + \norm{\A(x_i) - x_i}
        \\[5pt]
        \leq{}& 2\mathrm{dist}(\A(x_i), \X) + 2\norm{r_i} + \norm{\A(x_i) - x_i}\\
        \leq{}& \frac{8L_{x, \rho}M_{x, c}^2}{\sigma_{x, Q}^2} \norm{c(x_i)}^2 + 2M_{x, \xi} \norm{c(\xk)}^2 +  
        \frac{2 M_{x, Q} M_{x, c}}{ \sigma_{x, Q}} \norm{c(x_i)} \leq \frac{4 M_{x, Q} M_{x, c}}{ \sigma_{x, Q}} \norm{c(x_i)}.
    \end{aligned}
\end{equation}
Note that in deriving the second inequality, we use 
Lemma \ref{lemQ}, \eqref{Eq_Cond_rk}, and Lemma \ref{Le_aux_1}. 
Therefore, we have 
\begin{equation*}
    \begin{aligned}
        &\norm{\xkp - x_0} \leq \sum_{i = 0}^k \norm{x_{i+1} - x_i} \leq \frac{4 M_{x, Q} M_{x, c}}{ \sigma_{x, Q}} \sum_{i = 0}^k \norm{c(x_i)} \leq \frac{4 M_{x, Q} M_{x, c}}{ \sigma_{x, Q}} \norm{c(x_0)}\sum_{i = 0}^k \frac{1}{2^i} \\
        \leq{}& \frac{8 M_{x, Q} M_{x, c}}{ \sigma_{x, Q}} \norm{c(x_0)} \leq \frac{8 M_{x, Q} M_{x, c}^2}{ \sigma_{x, Q}} \mathrm{dist}(x_0, \M) \leq \frac{1}{2}\delta_x. 
    \end{aligned}
\end{equation*}
Together with the fact that $x_0 \in \Xi_x \subseteq \Omega_x$, we have $\norm{\xkp - x} \leq \norm{\xkp-x_0} + \norm{x_0 - x} \leq \delta_x$, which further illustrates that $\xkp \in \Omega_x$. Then we can conclude from the induction that $\{\xk\}$ lies within $\Omega_x$. This completes the proof.
\end{proof}

Finally, we present the following theorem illustrating the local quadratic convergence rate of \eqref{Eq_Feasibility_RestorationQ}. 
\begin{theo}
    \label{Theo_quadconv}
    Suppose Assumption \ref{Assumption_X} and Assumption \ref{Assumption_Q} hold. Then for any $x \in \K$ that satisfies the nondegeneracy condition, and any sequence $\{\xk\}$ generated by \eqref{Eq_Feasibility_RestorationQ} with $x_0\in\Xi_x$, the iterates $\{\xk\}$ converges to $\K \cap \Omega_x$ quadratically. 
\end{theo}
\begin{proof}
For any $x \in \K$ that satisfies the nondegeneracy condition, and any sequence $\{\xk\}$ generated by \eqref{Eq_Feasibility_RestorationQ} with $x_0\in\Xi_x$, Proposition \ref{prop:supconv} illustrates that $\{\xk\}$ is a Cauchy sequence and $\{\xk\} \subseteq \Omega_x$. 

Moreover, by using \eqref{eq-tmp1} in Lemma \ref{Le_quadratic_decrease}, we have 
    \begin{equation}
        \begin{aligned}
            \norm{c(\xkp)} \leq R_x\norm{c(\xk)}^2.
        \end{aligned}
    \end{equation}
    Then we can conclude from Lemma \ref{Le_aux_relationship_c_dist_PEJcc} that 
    \begin{equation}
        \label{Eq_Theo_quadconv_0}
        \mathrm{dist}(\xkp, \M) \leq 
        \frac{2M_{x,c}^2}{\sigma_{x, c}}\, R_x\,\mathrm{dist}(\xk, \M)^2. 
    \end{equation}
    
    Furthermore, for any $x \in \K$ that satisfies the nondegeneracy condition, it directly follows from \cite[Lemma 2.2]{xiao2025ongoing} that 
    \begin{equation}
        \{0\}\subseteq \ca{N}_{\M}(x) \cap \NX(x) \subseteq \ca{N}_{\M}(x) \cap \mathrm{range}(\NX(x)) = \{0\}.
    \end{equation}
   As demonstrated in \cite[Page 2]{drusvyatskiy2015transversality}, $\X$ and $\M$ intersects transversally at $x$, hence there exists $\hat{\delta}_{x} > 0$ and ${M}_{x, T}> 0$ 
    such that 
    \begin{equation}
        \mathrm{dist}(y, \K) \leq M_{x, T}\mathrm{dist}(y, \M), \quad \forall\; y \in \bb{B}(x;\hat{\delta}_x) \cap \X.
    \end{equation}
    Notice that $\K \subseteq \M$ implies that 
    $\mathrm{dist}(x, \M) \leq \mathrm{dist}(x,\K)$ for any $x \in \Rn$. 
    Together with \eqref{Eq_Theo_quadconv_0}, we can conclude that
    \begin{equation}
        \mathrm{dist}(\xkp, \K) \leq 
        \frac{2M_{x,c}^2M_{x, T}}{\sigma_{x, c}} \, R_x\,\mathrm{dist}(\xk, \K)^2. 
    \end{equation}
    This further demonstrates that $\{\mathrm{dist}(\xk, \K)\}$ converges to $0$ quadratically. This completes the proof. 
\end{proof}

\subsection{Globalization of the alternating projection with adaptive stepsizes}
\label{Subsection_global_alg}
In this subsection, we aim to develop a globally convergent alternating projection method based on the update scheme \eqref{Eq_Feasibility_RestorationQ}. 
Motivated by the equivalent formulation \eqref{Prob_Con}, we consider hybridizing \eqref{Eq_Feasibility_RestorationQ} with the projected gradient descent method for solving \eqref{Prob_Con}. In particular, in $k$-th iteration, we compute the trial point $x_{\mathrm{trial}, k}$ from $\xk$ based on the scheme \eqref{Eq_Feasibility_RestorationQ}, and measure the quality of $x_{\mathrm{trial}, k}$  by $\norm{c(x_{\mathrm{trial}, k})}$. When the quality of $x_{\mathrm{trial}, k}$ is poor (i.e., $\norm{c(x_{\mathrm{trial}, k})} \geq (1-\kappa)\norm{c(\xk)}$ for some prefixed $\kappa > 0$), we switch our algorithm to the projected gradient method for solving \eqref{Prob_Con} at $\xk$ and generate $\xkp$ accordingly. Otherwise, when the quality of $x_{\mathrm{trial}, k}$ is acceptable (i.e., $\norm{c(x_{\mathrm{trial}, k})} < (1-\kappa)\norm{c(\xk)}$, we set $\xkp = x_{\mathrm{trial}, k}$. The detailed algorithm is presented in Algorithm \ref{Alg:Adap_AP2}. 

\begin{algorithm}[htbp]
	\begin{algorithmic}[1]  
		\Require Closed  subset $\X$, mapping $c$, 
        $\kappa\in (0,1)$,
        $\eta_{\max} > 0$, $\alpha\in(0,1)$. 
		\State Set $k=0$;
		\While{not terminated}
            \State Compute $d_k = Q(\xk) \nabla c(\xk) (\nabla c(\xk)\tp Q(\xk) \nabla c(\xk) + \tau(\norm{c(\xk)}) I_p)^{-1} c(\xk)$
            \State Compute $x_{{\rm trial}, k} = \Pi_{\X}(\xk - d_k + r_k)$.
            \If{$\norm{c(x_{{\rm trial},k})} \geq (1-\kappa) \norm{c(\xk)}$}
                \State Find the smallest $j_k$ such that $\frac{1}{2}\norm{c(x_{+,k})}^2 \leq \frac{1}{2} \norm{c(\xk)}^2 - \frac{1}{4\alpha^{j_k}\eta_{\max}} \norm{x_{+,k} - \xk}^2$ with \begin{equation}\label{eq:pgl}x_{+, k} = \Pi_{\X}(\xk - \alpha^{j_k}\eta_{\max} \nabla c(\xk) c(\xk)). \end{equation}
                \State $\xkp = x_{+, k}$, 
                $\eta_k = \alpha^{j_k}\eta_{\max}.$
            \Else
                \State $\xkp = x_{{\rm trial},k}$
            \EndIf
            \State $k = k+1$  
		\EndWhile
		\State Return $\xk$.
	\end{algorithmic}  
	\caption{Adaptive alternating projection with hybrid projected gradient steps.}  
	\label{Alg:Adap_AP2}
\end{algorithm}

Now we prove the global convergence of Algorithm \ref{Alg:Adap_AP2}. We begin our theoretical analysis with the following lemma illustrating the effectiveness of the line-search procedure in Line 6 of Algorithm \ref{Alg:Adap_AP2}.
\begin{lem}
    \label{Le_sufficient_decrease}
    For any $x \in \X$ with $c(x)\neq0$, for any $\eta > 0$ such that 
    \begin{equation}\label{eq:eta}
        \eta\leq \min\left\{\frac{1}{2\norm{\nabla c(x)c(x)}}, \sup_{y, z\in \bb{B}(x;1)} \frac{\norm{\nabla c(y)c(y) - \nabla c(z)c(z)}}{4\norm{y-z}}\right\},
    \end{equation}
    we have 
    \begin{equation}\label{eq:nonin}
        \frac{1}{2} \norm{c(y_+)}^2 \leq \frac{1}{2} \norm{c(x)}^2 - \frac{1}{4\eta} \norm{y_+ - x}^2.
    \end{equation}
    Here $y_+ \in \Pi_{\X}(x - \eta \nabla c(x)c(x))$. %$x_{+} = \Pi_{\X}(x - \eta^2 d - \eta(1-\eta)\nabla c(x) c(x))$. 
\end{lem}
\begin{proof} For convenience, let $d = \nabla c(x)c(x)$.
    For any $\eta > 0$ such that  
    \begin{equation*}
        \eta\leq \min\left\{\frac{1}{2\norm{\nabla c(x)c(x)}}, \sup_{y, z\in \bb{B}(x;1)} \frac{\norm{\nabla c(y)c(y) - \nabla c(z)c(z)}}{4\norm{y-z}}\right\},
    \end{equation*}
    it follows from optimality of $y_+$ and $x\in\X$ that
    \begin{equation*}
        \norm{y_+ - x} = \norm{\Pi_{\X}(x - \eta d)-(x-\eta d) -\eta d}
        \leq \norm{\Pi_{\X}(x - \eta d)-(x-\eta d)}  + \eta\norm{d}
\leq
        2\eta \norm{d}\leq 1.
    \end{equation*}
    Moreover, it follows from the choice of $\eta$ that 
    \begin{equation}\label{eq:sc}
        \frac{1}{2} \norm{c(y_+)}^2 \leq \frac{1}{2} \norm{c(x)}^2 + \inner{\nabla c(x) c(x), y_+-x} + \frac{1}{4\eta}\norm{y_+ - x}^2. 
    \end{equation} 
    Also, from the property of the projection that 
    \begin{equation*}
        y_+ \in \arg\min_{y \in \X} \frac{1}{2} \norm{c(x)}^2 + \inner{\nabla c(x) c(x), y-x} + \frac{1}{2\eta}\norm{y - x}^2,
    \end{equation*}
    we have
    \begin{equation*}
        \frac{1}{2} \norm{c(x)}^2 \geq \frac{1}{2} \norm{c(x)}^2+ \inner{\nabla c(x) c(x), y_+-x} + \frac{1}{2 \eta} \norm{y_+-x}^2 \geq \frac{1}{2} \norm{c(y_+)}^2 + \frac{1}{4\eta}\norm{y_+ - x}^2.
    \end{equation*}
    This completes the proof. 
\end{proof}

Now we present the following proposition establishing the global convergence of Algorithm \ref{Alg:Adap_AP2}. 
\begin{prop}\label{prop:gloconv}
   Suppose Assumption \ref{Assumption_X} and Assumption \ref{Assumption_Q} hold, and $\X$. Then for the sequence $\{\xk\}$ generated by Algorithm \ref{Alg:Adap_AP2}, any of its cluster point $x^*$ satisfies $x^* \in \X$ and 
    \begin{equation}
        0 \in \nabla c(x^*)c(x^*) + \NX(x^*). 
    \end{equation}
    % where $\NX(x^*)$ is the regular normal cone \cite[Definition 6.3]{rockafellar2009variational} of $\X$ at $x^*$. 
\end{prop}
\begin{proof}
    It follows from Step 5--10 in Algorithm \ref{Alg:Adap_AP2} that   
    \begin{equation}
        \begin{aligned}
            &\frac{1}{2} \norm{c(\xkp)}^2 \leq \frac{1}{2} \norm{c(\xk)}^2 - \min\left\{ \frac{2\kappa-\kappa^2}{2} \norm{c(\xk)}^2, \frac{1}{4\eta_k}\norm{\Pi_{\X}(\xk - \eta_k \nabla c(\xk)c(\xk)) - \xk}^2 \right\}.
        \end{aligned}
    \end{equation}
    Then we can conclude that the sequence $\{\norm{c(\xk)}\}$ is non-increasing.

    Now let $\ca{I}$ be the indices of the iterations that Step 9 of Algorithm \ref{Alg:Adap_AP2} is performed. Suppose $\ca{I}$ contains infinitely many indices, then for any $k\geq 0$, 
    \begin{equation}
        \mathop{\lim\sup}_{k\to +\infty} \norm{c(\xk)} \leq \lim_{k\to +\infty} (1-\kappa)^{|\{i \in \ca{I}: i \leq k\}|} \norm{c(x_0)} = 0. 
    \end{equation}
    Therefore, we can conclude that $\lim_{k\to +\infty} \norm{c(\xk)} = 0$ and thus any limit point $x^*$ of $\{\xk\}$ satisfies $0 \in \nabla c(x^*) c(x^*) + \NX(x^*)$. 

    Conversely, suppose $\ca{I}$ only contains finitely many indices.  Let $N_{\ca{I}}$ be the largest index of $\ca{I}$. Then for any limit point $x^*$ of $\{\xk\}$, let $\{x_{i_k}\}$ be the subsequence of $\{\xk\}$ that converges to $x^*$ with $i_k \geq N_{\ca{I}}$ for any $k\geq 0$, then we have 
    \begin{equation}
        \frac{1}{2} \norm{c(x_{i_{j+1}})}^2 \leq \frac{1}{2} \norm{c(x_{N_{\ca{I}}})}^2 - \sum_{k = 0}^j\frac{1}{4\eta_{i_k}} \norm{\Pi_{\X}(x_{i_k} - \eta_{i_k} \nabla c(x_{i_k})c(x_{i_k})) - x_{i_k}}^2.
    \end{equation}
    Notice that since $\{x_{i_k}\}$ converges, it is uniformly bounded.
    As a result, from the boundedness of $\{x_{i_k}\}$ and \eqref{eq:eta}, it holds that $\inf_{k\geq 0} \eta_{i_k} > 0$, and thus 
    \begin{equation}
        \lim_{k\to +\infty} \Pi_{\X}(x_{i_k} - \eta_{i_k} \nabla c(x_{i_k})c(x_{i_k})) - x_{i_k} = 0.
    \end{equation}
    From the graph-closedness of $\NX(\cdot)$ and the necessary optimality  condition of \eqref{eq:pgl}, we can conclude that 
    \begin{equation}\label{eq:fonc}
        0 \in \nabla c(x^*) c(x^*) + \NX(x^*). 
    \end{equation}
    This completes the proof. 
\end{proof}  
As demonstrated in Proposition \ref{prop:gloconv}, Algorithm \ref{Alg:Adap_AP2} is only guaranteed to find the first-order stationary points of \eqref{Prob_Con}, which coincides with most existing results on the global convergence of LM methods. To enforce the convergence to $\K$, existing works require certain error-bound conditions. In the following corollary, we show that the error-bound condition ensures the global convergence to $\K$ for Algorithm \ref{Alg:Adap_AP2}. 

\begin{coro}
    Suppose Assumption \ref{Assumption_X} and Assumption \ref{Assumption_Q} hold, and the nondegeneracy condition, \eqref{Eq_Cond_Nondeg} holds for any $x \in \X$. Then the sequence $\{\xk\}$ generated by Algorithm \ref{Alg:Adap_AP2} converges to $x^* \in \K$ and achieves local quadratic convergence rate. 
\end{coro}
\begin{proof}
It is known from the nondegeneracy condition $\nabla c(x)\tp \mathrm{lin}(\ca{T}_{\X}(x)) = \Rp$ and \cite[Lemma 2.2]{xiao2025ongoing} that $\NX(x)\cap{\rm Range}(\nabla c(x))=\{0\}$.  Combining this with \eqref{eq:fonc}, we have $\nabla c(x^*)c(x^*)=0$, which implies $c(x^*)=0$ by using the nondegeneracy condition
$\nabla c(x^*)\tp \mathrm{lin}(\ca{T}_{\X}(x^*)) = \Rp$ again. It follows that $x^*\in\K$. The local quadratic rate follows directly from Theorem \ref{Theo_quadconv}. This completes the proof. 
\end{proof}

\subsection{Choosing projective mapping $Q$}\label{sec:Q}

As demonstrated in the theoretical analysis in Theorem \ref{Theo_quadconv}, the formulation of the projective mapping $Q$ plays an important role in constructing the locally quadratically convergent update scheme \eqref{Eq_Feasibility_RestorationQ}. Therefore, it is important to determine an explicit formulation of $Q$ for at least a wide range of closed subsets $\X$. 
Although the previous work \cite[Table 2]{xiao2025ongoing} presents several basic formulations of projective mappings for some specific convex compact sets $\X$, their formulations do not require the validity of Assumption \ref{Assumption_Q}(4).

Therefore, in the following, we present the formulation of the projective mapping $Q$ that satisfies the requirements in Assumption \ref{Assumption_Q} for possibly nonconvex subsets $\X$. Table \ref{Table_Q_mapping} presents the possible choices of the projective mapping $Q$ for some popular subsets $\X$. It can be 
verified that all the provided projective mappings satisfy all the requirements in Assumption \ref{Assumption_Q}.

Additionally, as an illustration, we present detailed verifications on the validity of Assumption \ref{Assumption_Q} for the projective mappings corresponding to the $\ell_q$-norm ball and low-rank variety. 
\begin{examples}[ $\ell_q$-norm ball ]
\label{Example_Q_lq}
When $\X$ is the $\ell_{q}$-norm ball in $\Rn$, i.e., $\X = \{x \in \Rn: \|x\|_q \leq 1\}$ for a prefixed $q \in (0,1]$, the corresponding projective mapping $Q$ can be chosen as 
        \begin{equation}
            Q(x) = \left(\mathrm{Diag}(|x|^{2-q}) - xx \tp \right) + (1 - \norm{x}_q^q) I_n. 
        \end{equation}
Denote $x:=(x_1,\dots,x_n)^\top$, it can be checked directly that Assumption \ref{Assumption_Q} (1) trivially holds. Also, it can be checked directly that 
\begin{equation*}
    \mathrm{Diag}(|x|^{q/2-1})\left(\mathrm{Diag}(|x|^{2-q}) - xx \tp \right)\mathrm{Diag}(|x|^{q/2-1})=I-ww^\top,
\end{equation*}
where $w=|x|^{\frac{q}{2}}\odot{\rm sign}(x)$, with $\odot$ denoting the Hadamard product. Since $\norm{w}_2^2=\norm{x}_q^q\leq1$, we know that $\mathrm{Diag}(|x|^{2-q}) - xx \tp $ is positive semi-definite. 
Combining this with $(1 - \norm{x}_q) I_n\in\mathbb{S}_+^n$, we know that $Q(x)$ is semi-definite for all $x\in\X$, which yields Assumption \ref{Assumption_Q} (2).

Next we verify the validity of Assumption \ref{Assumption_Q}(3). 
Let 
$$l(t) = \begin{cases}
    |t|^{q-1} \, \mathrm{sign}(t) & t\neq 0\\
    0 & t = 0.
\end{cases}$$
For any $x \in \X$ that satisfies $\norm{x}_q = 1$,
define $u(x) = [l(x_1), ..., l(x_n)]\tp$. 
Then we can conclude that $Q(x)u(x) = 0$. Moreover, for any $i \in [n]$ such that $x_i = 0$, it holds that $Q(x)e_i  = 0$ where $e_i \in \Rn$ refers to the vector whose $i$-th coordinate is $1$ and all the other coordinates are  $0$.
Notice that $\mathrm{rank}(Q(x)) = \mathrm{nnz}(x) -1$.
Define the vector $y = \sum_{\{i\in[n]: x_i=0\}} e_i\in\mathbb{R}^n$. Then we can conclude that 
\begin{equation*}
    \mathrm{null}(Q(x)) = \{t_1 u(x) + t_2 y:  (t_1, t_2) \in \bb{R}^2\},
\end{equation*}
which exactly equals to $\NX(x)$. 
Additionally, when $\norm{x}_{q} < 1$, it holds that $Q(x) \succ 0$, and hence $\mathrm{null}(Q(x)) =\{0\} = \NX(x)$. This completes the verification of Assumption \ref{Assumption_Q}(3).

Furthermore, notice that for any $t \neq 0$ and any $\delta \leq |t|$, it holds that $|t + \delta|^q \leq |t|^q + q \cdot \mathrm{sign}(t) |t|^{q-1} \delta $. Therefore, for any $x \in \X$ and any $d \in \Rn$ with $\norm{d} \leq n^{1/2-1/q}$, we have 
\begin{equation*}
    \norm{x + \left(\mathrm{Diag}(|x|^{2-q}) - xx \tp \right)d}_q^q \leq \norm{x}_q^q + q\inner{u(x), \left(\mathrm{Diag}(|x|^{2-q}) - xx \tp \right)d} = \norm{x}_q^q.
\end{equation*}
% $\norm{x + \left(\mathrm{Diag}(|x|^{2-q}) - xx \tp \right)d}_q^q \leq \norm{x}_q^q$. 
Noting that $(t_1 + t_2)^q \leq t_1^q + t_2^q$ holds for any $t_1, t_2 \in [0,1)$, we get
\begin{equation*}
    \begin{aligned}
        &\norm{x + Q(x)d}_q^q \leq \norm{x + \left(\mathrm{Diag}(|x|^{2-q}) - xx \tp \right)d}_q^q + \norm{(1-\norm{x}_q^q)d}_q^q \\
        \leq{}& \norm{x}_q^q + (1-\norm{x}_q^q) \norm{d}_q^q \leq \norm{x}_q^q + 1-\norm{x}_q^q \leq 1.
    \end{aligned}
\end{equation*}
Here the second last inequality uses the Jensen's inequality to get that $\norm{d}_q^q \leq n^{1-q/2}\norm{d}^q \leq 1$. 
Therefore, we can conclude that 
\begin{equation*}
     x+Q(x)d \in \X, \quad \forall\; \norm{d} \leq n^{1/2-1/q}.
\end{equation*}
On the other hand, for any $d\in\bb{R}^n$ satisfying $n^{1/2-1/q} \leq \norm{d} \leq 1$, let $\hat{d}_1 = n^{1/2-1/q}\frac{d}{\norm{d}}$ and $\hat{d}_2 = d - \hat{d}_1$. Then we can conclude that 
$\norm{\hat{d}_2} \leq \norm{d} \leq n^{1/q-1/2}\norm{d}^2$. Therefore, it follows from Lemma \ref{Le_aux_2} that 
\begin{equation*}
    \mathrm{dist}(x + Q(x)d, \X) =  \mathrm{dist}((x + Q(x)\hat{d}_1) + Q(x)\hat{d}_2 , \X) \leq 2 \norm{Q(x) \hat{d}_2} \leq 2 
    n^{1/q-1/2}\norm{Q(x)} \norm{d}^2. 
\end{equation*}
This verifies the validity of Assumption \ref{Assumption_Q}(4). 
 Therefore, we can conclude that $Q(x)$ satisfies all the conditions in Assumption \ref{Assumption_Q}.
\end{examples}

\begin{examples}[Low-rank variety]
When $\X \subseteq \bb{R}^{m \times q}$ ($m\geq q$) is the low-rank variety with rank $r$ ($0<r\leq q$), i.e., $\X := \{X \in \bb{R}^{m \times q}: \mathrm{rank}(X) \leq r\}$, we consider the following projective mapping $Q(X): \bb{R}^{m\times q}\to\bb{R}^{m\times q}$,
\begin{equation}
    Q(X)[D] = \frac{1}{2}( XX\tp D + DX\tp X).
\end{equation}
It can be checked directly that Assumption \ref{Assumption_Q} (1)-(2) hold. 

Moreover, for any $X \in \X$ that admits the singular value decomposition $X=[U_1,U_2]
\left[\begin{smallmatrix}
    \Sigma_1 & 0\\
    0 & 0
\end{smallmatrix}\right][V_1,\; V_2]^\top$ with $\mathrm{rank}(X)=d\leq r$, where $\Sigma_1\in \bb{R}^{d\times d}$, $U_1\in\bb{R}^{m\times d}$, $U_2\in\bb{R}^{m\times (m-d)}$, $V_1\in\bb{R}^{q\times d}$, $V_2\in\bb{R}^{q\times (q-d)}$,
it holds that $\NX(X)=\{Y\in\bb{R}^{m\times q} \mid Y =U_2MV_2^\top, 
M\in\bb{R}^{(m-d)\times(q-d)}\}=\mbox{range}(\NX(X))$. As a result, $\mbox{null}\big(Q(X)\big)=
\{D\in\bb{R}^{m\times q}: U_1\tp DV_1=0, \, U_2\tp D V_1 =0, ~ U_1\tp D V_2 = 0\}
=\NX(X)=\mbox{range}(\NX(X))$. This verifies the validity of Assumption \ref{Assumption_Q}(3).

Let $U=[U_1,U_2]$ and $V=[V_1,V_2]$.
For any $X\in\X$, and any $D\in\mathbb{R}^{m\times q}$, it holds that 
\begin{equation}
    \begin{aligned}
            &U\tp (X+Q(X)[D]) V  
    = 
    \left[\begin{matrix}
    \Sigma_1 + \frac{1}{2}\Sigma_1^2 U_1\tp D V_1 + \frac{1}{2}U_1\tp D V_1  \Sigma_1^2& \frac{1}{2}\Sigma_1^2 U_1\tp D V_2\\[3pt]
    \frac{1}{2}U_2\tp D V_1 \Sigma_1^2 & 0
\end{matrix}\right].
    \end{aligned}
\end{equation}

Let
\begin{equation*}
    \begin{aligned}
&Y = U\left[\begin{matrix}
    I_d& 0\\
    \frac{1}{2}U_2\tp D V_1 \Sigma_1&0
\end{matrix}\right]\left[\begin{matrix}
    \Sigma_1 + \frac{1}{2}\Sigma_1^2 U_1\tp D V_1 + \frac{1}{2}U_1\tp D V_1  \Sigma_1^2& \frac{1}{2}\Sigma_1^2 U_1\tp D V_2\\
    0&0
\end{matrix}\right] V^\top.
    \end{aligned}
\end{equation*}
We have that ${\rm rank}(Y) \leq d$, and hence $Y\in \X$.
Then from the expression of $U\tp (X+Q(X)[D])V$, we have
\begin{equation*}
    \begin{aligned}
    &\norm{ (X+Q(X)[D])  - Y}
        =\norm{U\tp (X+Q(X)D)V  - U^\top Y V}\\
    ={}&\frac{1}{4}\norm{\left[\begin{matrix}
    0& 0\\
     U_2\tp D V_1 \Sigma_1(\Sigma_1^2 U_1\tp D V_1 + U_1\tp D V_1\Sigma_1^2)& U_2\tp D V_1 \Sigma_1^3 U_1\tp D V_2
\end{matrix}\right]} \leq \frac{3}{4}\norm{X}^3 \norm{D}^2. 
    \end{aligned}
\end{equation*}
Therefore, we can conclude that 
\begin{equation*}
    \mathrm{dist}(X + Q(X)[D], \X) \leq \norm{X + Q(X)[D] - Y}
    \leq 
    \frac{3}{4}\norm{X}^3\norm{D}^2, \quad \forall ~ \norm{D} \leq 1.
\end{equation*}
This illustrates that  Assumption \ref{Assumption_Q}(4) holds. Therefore, we can conclude that $Q$ satisfies all the conditions in Assumption \ref{Assumption_Q}.

\end{examples}

   \begin{table}[tb]
    \centering
    \small
    \begin{tabular}{p{4cm}|p{5cm}|p{6cm}}
        \hline
        \textbf{Name of constraints} & \textbf{Formulation of $\X$} & \textbf{Possible choices of projective mapping} \\ \hline
        Box constraints &
        $\{x \in \Rn: l \leq x \leq u\}$ &
        $Q(x) = \mathrm{Diag}(q_1(x), \ldots, q_n(x))$ \newline where $\{q_i\}$ are Lipschitz continuous, $q_i(l_i) = q_i(u_i) = 0$, and $q_i(s) > 0$ for $s \in (l_i, u_i)$ \\ \hline

        Nonnegative orthant &
        $\{x \in \Rn: x \geq 0\}$ &
        $Q(x) = \mathrm{Diag}(x)$ \\[3pt] \hline
        
        $\ell_2$-norm ball &
        $\{x \in \Rn: \norm{x} \leq u\}$ &
        $Q(x) = I_n - \frac{xx\tp}{u^2}$ \\[3pt] \hline
        
        $\ell_1$-norm ball &
        $\{x \in \Rn: \norm{x}_1 \leq 1\}$ &
        $Q(x) = \left(\text{Diag}(|x|) - xx\tp \right) + (1-\norm{x}_1) I_n$ \\ \hline

        Probability simplex &
        $\{x \in \Rn: x \geq 0, \norm{x}_1 = 1\}$ &
        $Q(x) = \text{Diag}(x) - xx\tp $ \\[3pt] \hline

        Spectral constraints &
        $\{X \in \bb{R}^{m\times s} : \norm{X}_2 \leq 1\}$ &
        $Q(X)= Y \mapsto Y - X \Phi(X\tp Y)$ \\[3pt] \hline
        
        PSD cone &
        $\{X \in \bb{R}^{s\times s} : X \succeq 0\}$ &
        $Q(X)= Y \mapsto \Phi(XY)$ \\[3pt] \hline
        
        PSD matrices with spectral constraints &
        $\{X \in \bb{R}^{s\times s} : X \succeq 0, \norm{X}_2 \leq 1\}$ &
        $Q(X)= Y \mapsto \Phi(X(I_n - X)Y)$ \\ \hline

        $l_0$-norm ball & 
        $\{x \in \Rn: \|x\|_0 \leq u\}$ &
        $Q(x) = \mathrm{Diag}(x^{\odot 2})$ \\[3pt] \hline

        $l_q$-norm ball with $0<  q\leq 1$& 
        $\{x \in \Rn: \|x\|_q \leq 1\}$ &
        $Q(x) = \left( \Diag(|x|^{2-q}) - x x^\top \right) + (1 - \|x\|_q^q) I_n$ \\ \hline

        Low-rank variety & 
        $\{X \in \bb{R}^{m \times q}: \mathrm{rank}(X) \leq r\}$ &
        $Q(X) = D \mapsto \frac{1}{2}\left(DX\tp X +  XX\tp D \right)$ \\ \hline

        Symmetric low-rank variety & 
        $\{X \in \bb{S}^{m \times m}: \mathrm{rank}(X) \leq r\}$ &
        $Q(X) = D \mapsto\Phi(DX^2)$ \\ \hline

        Low-rank PSD matrices & 
        $\{X \in \bb{S}^{m \times m}: \mathrm{rank}(X) \leq r, ~X \succeq 0\}$ &
        $Q(X) = D \mapsto\Phi(DX)$ \\ \hline
    \end{tabular}
    \caption{Some closed and possibly nonconvex subsets $\X$ and their corresponding projective mappings. Here $\Phi$ is the symmetrization of a square matrix, defined as $\Phi(M):= \frac{M+M\tp}{2}$. Additionally, $\bb{S}^{m \times m}$ denotes the set of real $m \times m$ symmetric matrices. }
    \label{Table_Q_mapping}
\end{table}

\section{A Quadratically Convergent Bregman Proximal Method}

In this section, we investigate the relationship between \eqref{Eq_Feasibility_RestorationQ} and the Bregman proximal gradient methods.

We begin with the following conditions on the Bregman kernel function. Let $\phi: \Rn \to \bb{R} \cup \{+\infty\}$ be a twice continuously differentiable function over its domain, we say $\phi$ is a projective Bregman kernel function for $\X$ if it satisfies the following conditions. 
\begin{cond}
    \label{Cond_Bregman}
    \begin{enumerate}
        \item $\phi$ is proper, lower-semicontinuous and strictly convex over $\Rn$ with ${\rm dom}(\phi):=\{x \in \Rn: \phi(x) < +\infty\} = \mathrm{int}(\X)$.
        \item $\phi$ is three-times continuously differentiable over $\mathrm{int}(\X)$, and the following function
        \begin{equation}
             T_{\phi}(x;d) := \frac{1}{2} \nabla^2 \phi(x)^{-1} \big(\nabla^3 \phi(x)[\nabla^2 \phi(x)^{-1}d] \big)\nabla^2 \phi(x)^{-1}
        \end{equation}
         is locally bounded over $\X \times \Rn$.
        \item The value of $(\nabla^2 \phi(x))^{-1}$ can be naturally extended from $\mathrm{int}(\X)$ to $\X$, which is denoted as $W_{\phi}(x)$. Moreover, $W_{\phi}(x)$ is locally Lipschitz continuous over $\X$, and $\mathrm{null}(W_{\phi}(x)) = \mathrm{range}(\ca{N}_{\X}(x))$ for any $x \in \X$. 
    \end{enumerate}
\end{cond}
From Condition \ref{Cond_Bregman}, we can conclude that for any $\phi$ that is a projective Bregman kernel function for $\X$, it is a Bregman kernel function for $\X$. Therefore, based on the formulation of the Bregman kernel function, we first consider the following update scheme,
\begin{equation}
    \label{Eq_Bregman_PG}
    \xkp = \mathop{\arg\min}_{y \in \Rn}\quad \inner{y-\xk, d(\xk)} + D_{\phi}(y, \xk),
\end{equation}
where $d(\xk)$ is defined as in \eqref{Eq_AQ_general}, $D_{\phi}(y, x):= \phi(y) - \phi(x) - \inner{\nabla \phi(x), y-x}$ is the Bregman divergence induced by $\phi$. Moreover, based on the differentiability of $\phi$ and the optimality condition of the subproblem  \eqref{Eq_Bregman_PG}, we can reformulate \eqref{Eq_Bregman_PG} as the following scheme,
\begin{equation}
    \label{Eq_Bregman_PG_inv}
    \xkp = (\nabla \phi)^{-1}\left( \nabla \phi(\xk) - d(\xk)  \right). 
\end{equation}

Next, from the local boundedness of $T_{\phi}(x;d)$ in Condition \eqref{Cond_Bregman}, we define the following constant, 
\begin{equation}
    L_{x, T} := \sup_{\norm{y-x}\leq \norm{\nabla^2 \phi(x)^{-1}}}\; \sup_{d \in \Rn, \norm{d} = 1} \norm{T_{\phi}(y;d)}. 
\end{equation} 
%%Then the following lemma illustrates the well-definedness of $L_{x, T}$. 
\begin{lem}
    \label{Le_approx_inv_nablaphi}
    Suppose $\phi$ is a projective Bregman kernel function for $\X$. Then for any $x \in \mathrm{int}(\X)$ and any $d \in \bb{B}(0;1)$, it holds that 
    \begin{equation*}
        \norm{(\nabla \phi)^{-1}(\nabla \phi(x) + d) - \left( x + \nabla^2 \phi(x)^{-1}d  \right)} \leq L_{x, T}\norm{d}^2. 
    \end{equation*}
\end{lem}
\begin{proof}
    From the smoothness of $\phi$ and the chain-rule applied
    to $(\nabla \phi)^{-1} \nabla\phi(x) = x$, we can conclude that for any $x \in \mathrm{int}(\X)$, it holds that 
    \begin{equation*}
        \left(\nabla (\nabla \phi)^{-1} \right) (\nabla \phi(x)) = \nabla^2 \phi(x)^{-1}. 
    \end{equation*}
    As a result, by applying the chain-rule again, we can conclude that for any $d \in \Rn$, it holds that 
    \begin{equation*}
        \left(\nabla^2 (\nabla \phi)^{-1} \right) (\nabla \phi(x))[d] = -\nabla^2 \phi(x)^{-1} \left(\nabla^3 \phi(x)[\nabla^2 \phi(x)^{-1}d] \right) \nabla^2 \phi(x)^{-1}.
    \end{equation*}
    As a result, we can conclude that 
    \begin{equation*}
        \norm{(\nabla \phi)^{-1}(\nabla \phi(x) + d) - \left( x + \nabla^2 \phi(x)^{-1}d - T_{\phi}(x; d) d \right)} = o(\norm{d}^2). 
    \end{equation*}
    Together with the mean-value theorem, we can conclude that 
    \begin{equation*}
        \norm{(\nabla \phi)^{-1}(\nabla \phi(x) + d) - \left( x + \nabla^2 \phi(x)^{-1}d  \right)} \leq L_{x, T}\norm{d}^2. 
    \end{equation*}
    This completes the proof. 
\end{proof}

Moreover, we have the following proposition illustrating that $\nabla^2 \phi(x)^{-1}$ is a projective mapping for $\X$.
\begin{prop}
    \label{Prop_Hessinv_Projjective}
    Suppose Assumption \ref{Assumption_X} holds, and $\phi$ is a projective Bregman kernel function for $\X$, then the mapping $W_{\phi}(x)$ is a projective mapping for $\X$.
\end{prop}
\begin{proof}
    The validity of Assumption \ref{Assumption_Q}(1)-(3) can be verified by direct calculation. Therefore, we only need to verify the validity of Assumption \ref{Assumption_Q}(4). 

    Notice that for any $x \in \mathrm{int}(\X)$ and any $d \in \bb{B}(0;1)$, there exists $\alpha \in [0, 1]$ and $z = (\nabla \phi)^{-1}(\nabla \phi(x) + \alpha d)$ such that 
    \begin{equation}
        \begin{aligned}
            &\mathrm{dist}(x + W_{\phi}(x) d, \X) = \mathrm{dist}(x + \nabla^2 \phi(x)^{-1} d, \X) \\
            \leq{}& \norm{ (x + \nabla^2 \phi(x)^{-1} d) - (\nabla \phi)^{-1}(\nabla \phi(x) + d) } \leq L_{x,T}\norm{d}^2.
        \end{aligned}
    \end{equation}
    Here the last inequality follows from Lemma \ref{Le_approx_inv_nablaphi}. 
    Therefore, we can conclude the validity of Assumption \ref{Assumption_Q}(4). This completes the proof. 
\end{proof}

Then based on Theorem \ref{Theo_quadconv}, the following theorem illustrates that the iterates $\{\xk\}$ generated by \eqref{Eq_Bregman_PG} fits into the framework \eqref{Eq_Feasibility_RestorationQ}, which further leads to the local quadratic convergence rate of $\{\xk\}$.
\begin{theo}
    \label{Theo_Bregman_quadconv}
    Suppose Assumption \ref{Assumption_X} holds, $\phi$ is a projective Bregman kernel function. Then for any $x \in \K$ satisfying the nondegeneracy condition \eqref{Eq_Cond_Nondeg} and any $x_0 \in \Xi_x \cap \mathrm{int}(\X)$, the iterates $\{\xk\}$ generated by \eqref{Eq_Bregman_PG} 
    converges to $x^* \in \Omega_x\cap\K$ quadratically. 
\end{theo}
\begin{proof}
    For any $x \in \X$ that satisfies the nondegeneracy condition, and  any $y \in \Omega_x$, let
    \begin{equation*}
        \begin{aligned}
            r_{B}(y) ={}&  \Big(y- \nabla^2 \phi(y)^{-1} d(y)\Big)
                        -  (\nabla \phi)^{-1}\left( \nabla \phi(y) - d(y) \right).
        \end{aligned}
    \end{equation*}
    where $d(y)$ is defined as in \eqref{Eq_AQ_general}.
    Then the update scheme of $\{\xk\}$ generated by \eqref{Eq_Bregman_PG} can be expressed as 
    \begin{equation*}
        \xkp = \Pi_{\X}\Big(\xk - \nabla^2 \phi(\xk)^{-1}d(\xk) - r_{B}(\xk)\Big). 
    \end{equation*}
     It follows from Lemma \ref{Le_approx_inv_nablaphi} that, for any $y \in \Omega_x$, we have
    \begin{equation}
        \label{Eq_Prop_Bregman_PG_Near_0}
        \begin{aligned}
            r_{B}(y) \leq L_{x, T} \norm{d(y)}^2 \leq \frac{4L_{x,T}M_{x, Q}^2M_{x, c}^2}{\sigmaxc^2} \norm{c(y)}^2.
        \end{aligned} 
    \end{equation}
    Therefore, for any iterates generated by \eqref{Eq_Bregman_PG} with $x_0 \in \Xi_{x} \cap \mathrm{int}(\X)$, the update scheme of $\{\xk\}$ follows from \eqref{Eq_Feasibility_RestorationQ} with $r_k = r_B(\xk)$. Moreover, it directly follows from \eqref{Eq_Prop_Bregman_PG_Near_0} that $L_{x, \xi} \leq \frac{4L_{x,T}M_{x, Q}^2M_{x, c}^2}{\sigmaxc^2}$. 
    Therefore, Theorem \ref{Theo_quadconv} demonstrates the local quadratic convergence of $\{\xk\}$. This completes the proof. 
\end{proof}

\section{Numerical experiments}\label{sec:num}

\begin{table}[tb]
	\centering
	\caption{Comparison with alternating projection method for solving \eqref{Example_CorrSparse}.}
	\label{Table_num1_lr1}
	\begin{tabular}{cc|ccc}
		\hline
		&           & iter & feas & time (s) \\ \hline
        \multicolumn{1}{l|}{\multirow{2}{*}{$n=100$}}& APHL   & 27  & 1.79e-12  & 1.66\\
		\multicolumn{1}{l|}{}                      & APM      & 205  & 9.09e-11  & 0.50\\
		\hline
		\multicolumn{1}{l|}{\multirow{2}{*}{$n=200$}}& APHL   & 35  & 1.27e-13  & 2.12\\
		\multicolumn{1}{l|}{}                      & APM      & 328  & 9.41e-11  & 2.19\\
		\hline
        \multicolumn{1}{l|}{\multirow{2}{*}{$n=500$}}& APHL   & 56  & 3.15e-11  & 4.77\\
		\multicolumn{1}{l|}{}                    & APM  & 615  & 9.66e-11  & 25.63\\
		\hline
        \multicolumn{1}{l|}{\multirow{2}{*}{$n=1000$}}& APHL   & 82  & 2.24e-13  & 28.97\\
		\multicolumn{1}{l|}{}                    & APM  & 781  & 9.99e-11  & 159.84\\
		\hline
        \multicolumn{1}{l|}{\multirow{2}{*}{$n=1500$}}& APHL   & 102  & 1.50e-13  & 151.03\\
		\multicolumn{1}{l|}{}                    & APM  & 794  & 9.71e-11  & 427.56\\
		\hline
        \multicolumn{1}{l|}{\multirow{2}{*}{$n=2000$}}& APHL   & 120  & 5.06e-14  & 294.01\\
		\multicolumn{1}{l|}{}                    & APM  & 965  & 9.85e-11  & 1035.04\\
		\hline
	\end{tabular}
\end{table}

\begin{figure}[htbp]
    \centering
    % 第一行 Case 1
    \caption{The performance curves of Algorithm \ref{Alg:Adap_AP2} and APM for solving \eqref{Example_CorrSparse}.}
    \label{fig:cor}
    \subfigure[$n=100$]{
        \begin{minipage}[t]{0.30\textwidth}
            \centering
            \includegraphics[width=\textwidth]{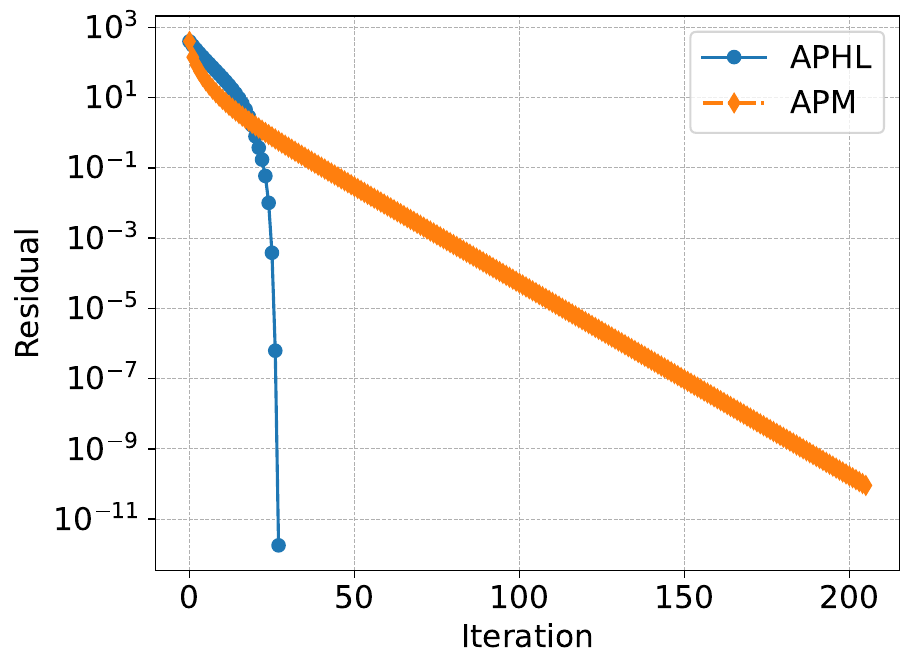}
        \end{minipage}
    }
    \subfigure[$n=200$]{
        \begin{minipage}[t]{0.30\textwidth}
            \centering
            \includegraphics[width=\textwidth]{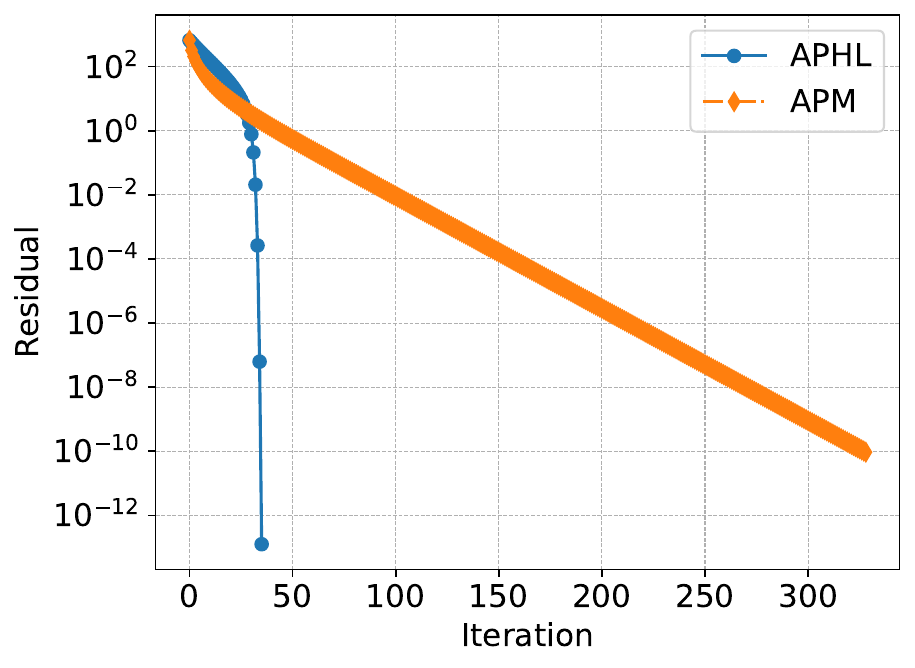}
        \end{minipage}
    }
    \subfigure[$n=500$]{
        \begin{minipage}[t]{0.30\textwidth}
            \centering
            \includegraphics[width=\textwidth]{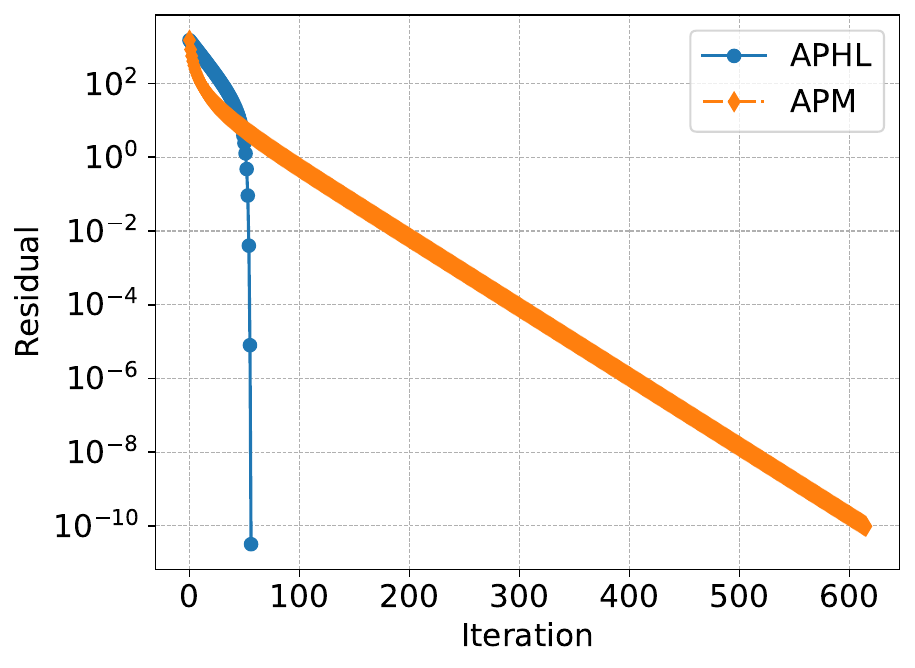}
        \end{minipage}
    }
    
    \subfigure[$n=1000$]{
        \begin{minipage}[t]{0.30\textwidth}
            \centering
            \includegraphics[width=\textwidth]{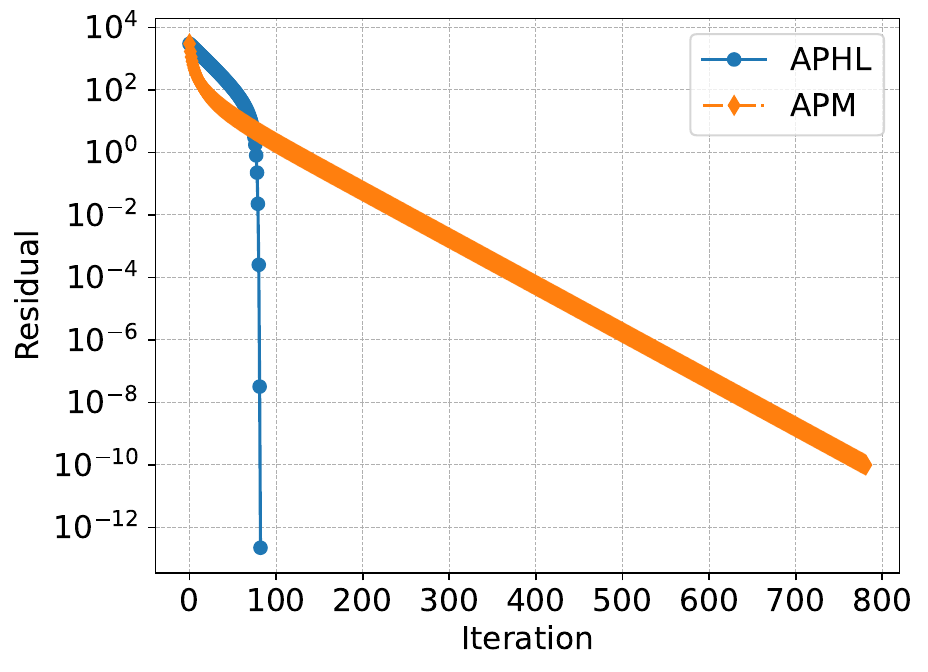}
        \end{minipage}
    }
    \subfigure[$n=1500$]{
        \begin{minipage}[t]{0.30\textwidth}
            \centering
            \includegraphics[width=\textwidth]{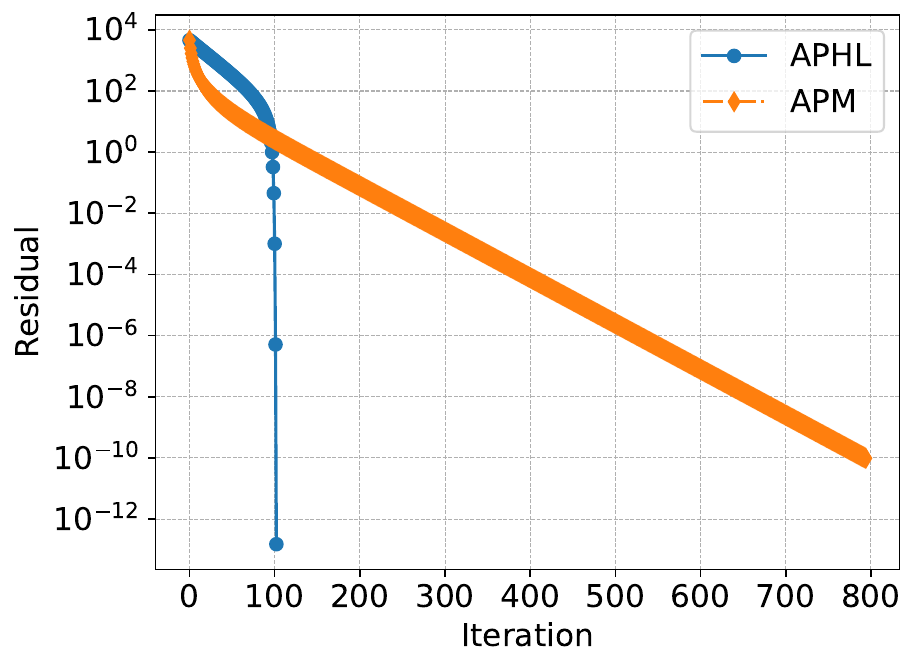}
        \end{minipage}
    }
    \subfigure[$n=2000$]{
        \begin{minipage}[t]{0.30\textwidth}
            \centering
            \includegraphics[width=\textwidth]{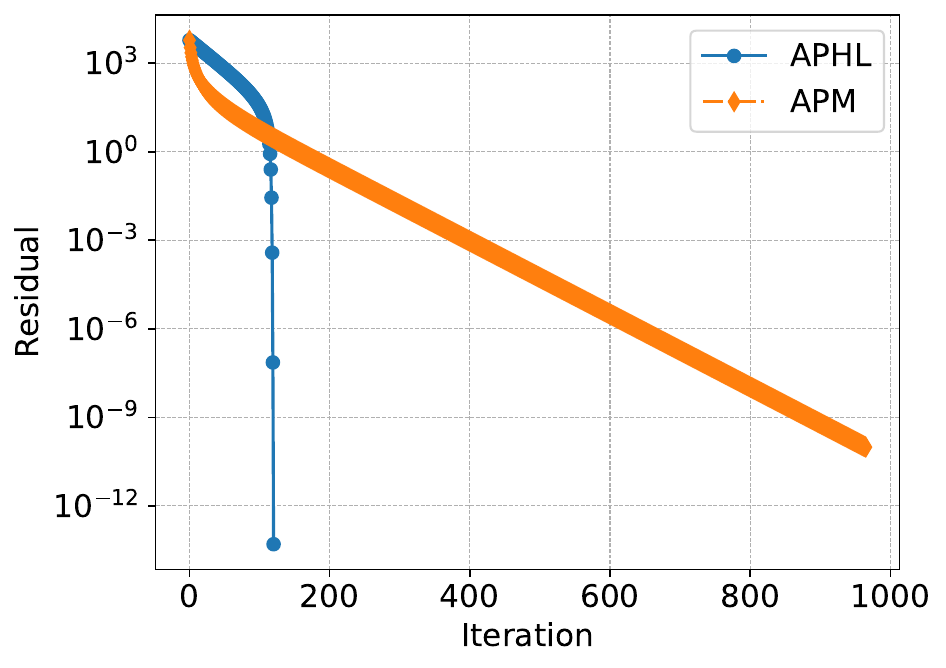}
        \end{minipage}
    }
    
\end{figure}

In this section, several numerical experiments are conducted to show the quadratic convergence rate of our proposed adaptive alternating projection method, which is consistent with the theory established in Section \ref{seca}. Comparisons with existing solvers on specific problems are also presented to exhibit the efficiency and robustness of our proposed algorithm. 
All experiments are conducted in Python 3.13.3 on a macOS server equipped with an M2 chip and 8 GB of RAM. 

In each test instance of our numerical experiments, we set $\eta_{\max} = 1$, $\alpha = 0.7$, and the maximum number of line-search steps as $10$ for Algorithm \ref{Alg:Adap_AP2}. Moreover, we terminate Algorithm \ref{Alg:Adap_AP2} when $\norm{c(\xk)} \leq 10^{-10}$ or the maximum number of iterations exceed $5000$.

%%%%%%%%%%%%%%%%%%%%%%%%%%%%%%%%%
\subsection{Sparse correlation matrix}

In this subsection, we test the performance of Algorithm \ref{Alg:Adap_AP2} on the following sparse correlation matrix problem: 
\begin{equation}
    \label{Example_CorrSparse}
    \mathrm{Find} \quad X \in \bb{S}^n_+, \quad \text{s. t.} \quad {\rm diag}(X)=\mathbf{1}_n, ~\inner{E_{ij}, X} = 0, ~ \forall ~(i,j) \in \ca{I}. 
\end{equation}
Here ${\bf 1}_n$ denotes the $n-$dimensional vector of all ones, $E_{ij}\in\mathbb{R}^{n\times n}$ is the matrix with the $(i,j)$-th entry being 1, and all other entries being $0$, and $\ca{I}$ is a subset of $\{(i,j): i\in [n], j\in [n]\}$. 

In our numerical experiments,  we randomly generate $W_{\mathrm{ref}} \in \bb{R}^{n \times n}$ by the ``sparse.random" function in SciPy with density equals to $0.1$ for $n\leq1000$ and $0.05$ otherwise. Then we normalize the column vectors of $W_{\mathrm{ref}}$ to unit norm (hence $\mathrm{diag}(W_{\mathrm{ref}}\tp W_{\mathrm{ref}}) = \mathbf{1}_n$), and generate the reference point $X_{\mathrm{ref}} = W_{\mathrm{ref}}\tp W_{\mathrm{ref}}$. We choose $\ca{I}$ as the indices of  the zero  entries of $X_{\mathrm{ref}}$. Then we choose the initial point $X_0 = \Phi(10 \cdot \mathrm{randn}(n,n)) + X_{\mathrm{ref}}$. 

We compare the numerical performance of Algorithm \ref{Alg:Adap_AP2} (denote as APHL) with the standard alternating projection method \cite{drusvyatskiy2014alternating} (denote as APM). In the APM method, the iterates are alternatively projected to $\X$ and $\M$, i.e., $\xkp = \Pi_{\M}( \Pi_{\X}(\xk))$. The maximum iteration of APM is set to be $5000$. 

Our numerical results are presented in Table \ref{Table_num1_lr1}, and the corresponding performance curves are presented in Figure \ref{fig:cor}. From these numerical results, we can conclude that both Algorithm \ref{Alg:Adap_AP2} and APM can find a feasible point of \eqref{Example_CorrSparse} with high precision.
More importantly, Algorithm \ref{Alg:Adap_AP2} takes significantly less CPU time than APM to solve the test instances with $n > 500$. 
Moreover, as demonstrated in Figure \ref{fig:cor}, Algorithm \ref{Alg:Adap_AP2} demonstrates a quadratic convergence rate, while APM only demonstrates a linear convergence rate. These empirical performances of Algorithm \ref{Alg:Adap_AP2} are consistent with our theoretical results in Theorem \ref{Theo_quadconv}, and demonstrate the high computational efficiency of Algorithm \ref{Alg:Adap_AP2}.

\subsection{Low-rank variety with constraints}
\label{Subsection_lr}

\begin{table}[tb]
	\centering
    \footnotesize
	\caption{The comparison between Algorithm \ref{Alg:Adap_AP2} and NewtonSLRA for solving \eqref{Example_lr_variety}.}
	\label{Table_num1_lr}
	\begin{tabular}{cc|ccc}
		\hline
		&           & iter & feas & time (s) \\ \hline
		\multicolumn{1}{l|}{\multirow{2}{*}{$(n,m,p,r)=(100,100,500,80)$}}& APHL   & 3  & 7.37e-11  & 0.25\\
		\multicolumn{1}{l|}{}                      & NewtonSLRA      & 3  & 2.36e-12  & 277.06\\
		\hline
        \multicolumn{1}{l|}{\multirow{2}{*}{$(n,m,p,r)=(100,100,10,80)$}}& APHL   & 3  & 9.96e-12  & 0.03\\
		\multicolumn{1}{l|}{}                      & NewtonSLRA      & 2  & 2.52e-12  & 188.43\\
		\hline
        \multicolumn{1}{l|}{\multirow{2}{*}{$(n,m,p,r)=(100,100,200,10)$}}& APHL   & 4  & 2.05e-11  & 0.09\\
		\multicolumn{1}{l|}{}                      & NewtonSLRA & 4  & 1.33e-12  & 18.68      \\
		\hline
        \multicolumn{1}{l|}{\multirow{2}{*}{$(n,m,p,r)=(1000,100,500,80)$}}& APHL   & 6  & 7.87e-11  & 2.87\\
		\multicolumn{1}{l|}{}                      & NewtonSLRA & -  & -  & -      \\
		\hline
        \multicolumn{1}{l|}{\multirow{2}{*}{$(n,m,p,r)=(1000,100,500,10)$}}& APHL   & 7  & 3.49e-11  & 4.37\\
		\multicolumn{1}{l|}{}                      & NewtonSLRA & -  & -  & -      \\
		\hline
        \multicolumn{1}{l|}{\multirow{2}{*}{$(n,m,p,r)=(100,1000,50,80)$}}  & APHL & 6  & 1.54e-11  & 0.59\\
		\multicolumn{1}{l|}{}                      & NewtonSLRA & -  & -  & -      \\
		\hline
        \multicolumn{1}{l|}{\multirow{2}{*}{$(n,m,p,r)=(1000,1000,50,900)$}}  & APHL & 8  & 2.43e-12  & 20.63\\
		\multicolumn{1}{l|}{}                      & NewtonSLRA & -  & -  & -      \\
		\hline
        \multicolumn{1}{l|}{\multirow{2}{*}{$(n,m,p,r)=(1000,1000,50,100)$}}  & APHL & 9  & 1.62e-12  & 6.11\\
		\multicolumn{1}{l|}{}                      & NewtonSLRA & -  & -  & -      \\
		\hline
        \multicolumn{1}{l|}{\multirow{2}{*}{$(n,m,p,r)=(1000,1000,500,100)$}}  & APHL & 9  & 3.24e-11  & 57.46\\
		\multicolumn{1}{l|}{}                      & NewtonSLRA & -  & -  & -      \\
		\hline
	\end{tabular}
\end{table}

\begin{figure}[htbp]
    \centering
    % 第一行 Case 1
    \caption{The performance curves of Algorithm \ref{Alg:Adap_AP2} and NewtonSLRA for solving \eqref{Example_lr_variety}.}
    \label{fig:lr}
    \subfigure[$(n,m,p,r)=(100,100,500,80)$]{
        \begin{minipage}[t]{0.30\textwidth}
            \centering
            \includegraphics[width=\textwidth]{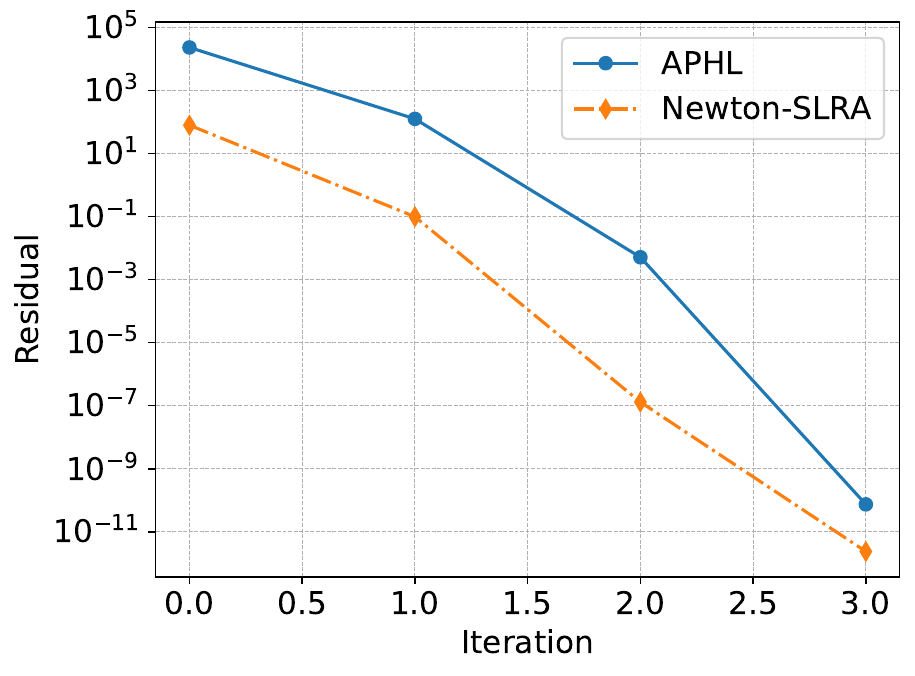}
        \end{minipage}
    }
    \subfigure[$(n,m,p,r)=(100,100,10,80)$]{
        \begin{minipage}[t]{0.30\textwidth}
            \centering
            \includegraphics[width=\textwidth]{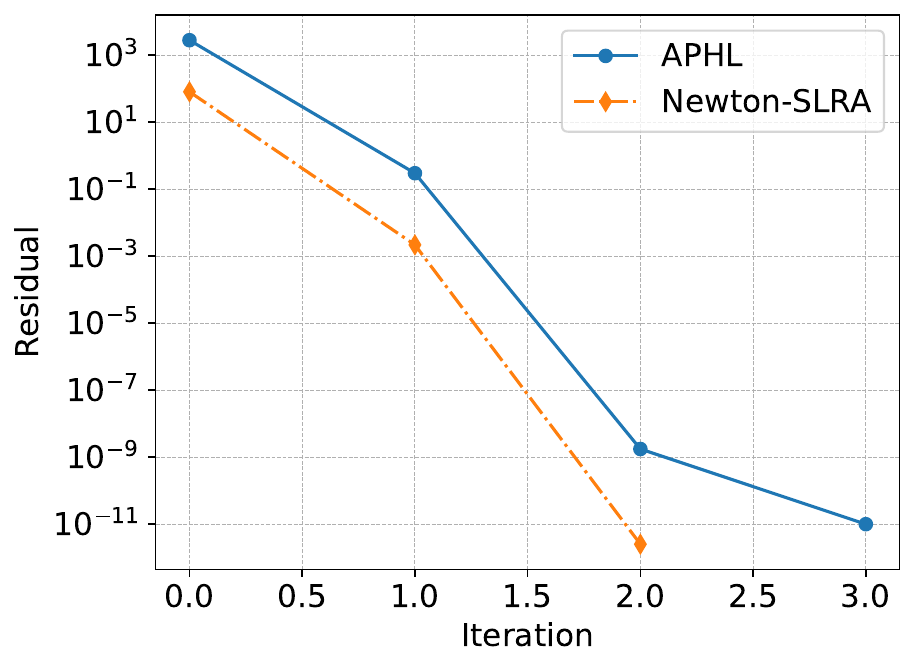}
        \end{minipage}
    }
    \subfigure[$$(n,m,p,r)=(100,100,200,10)$$]{
        \begin{minipage}[t]{0.30\textwidth}
            \centering
            \includegraphics[width=\textwidth]{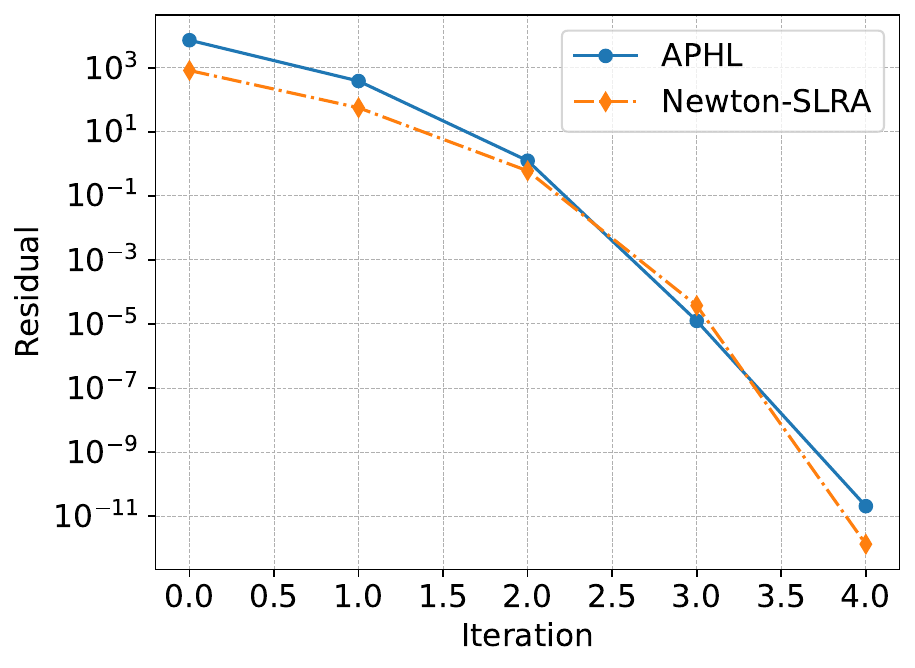}
        \end{minipage}
    }
    
    \subfigure[$(n,m,p,r)=(1000,100,500,80)$]{
        \begin{minipage}[t]{0.30\textwidth}
            \centering
            \includegraphics[width=\textwidth]{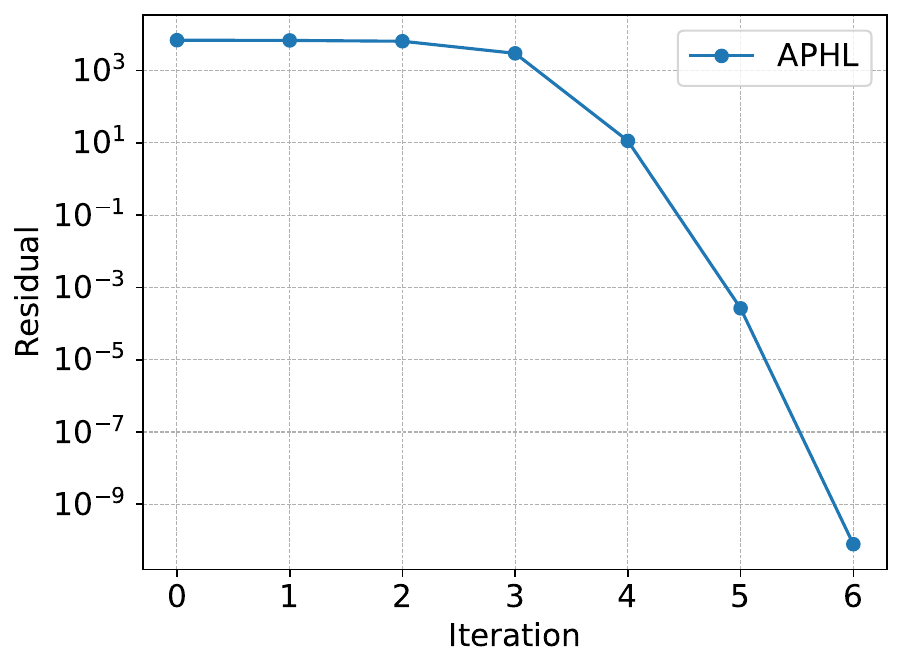}
        \end{minipage}
    }
    \subfigure[$$(n,m,p,r)=(1000,100,500,10)$$]{
        \begin{minipage}[t]{0.30\textwidth}
            \centering
            \includegraphics[width=\textwidth]{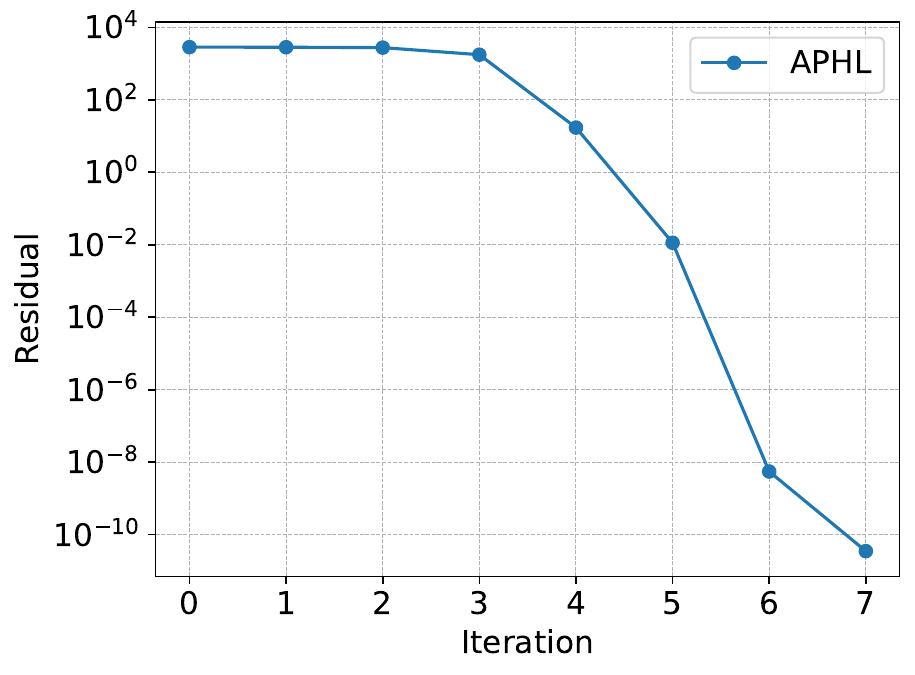}
        \end{minipage}
    }
    \subfigure[$$(n,m,p,r)=(100,1000,50,80)$$]{
        \begin{minipage}[t]{0.30\textwidth}
            \centering
            \includegraphics[width=\textwidth]{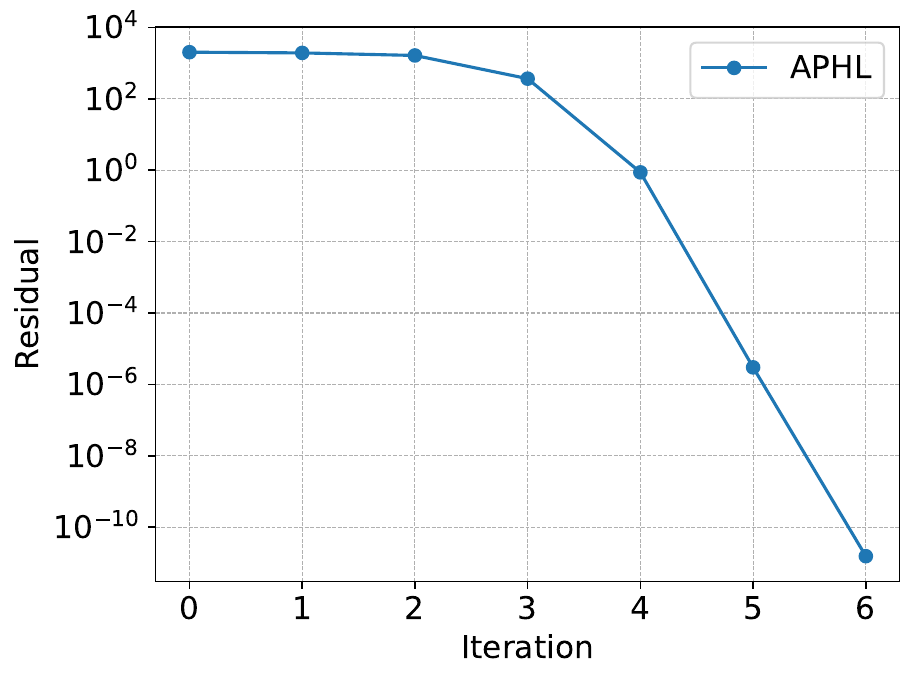}
        \end{minipage}
    }

    \subfigure[$(n,m,p,r)=(1000,1000,50,900)$]{
        \begin{minipage}[t]{0.30\textwidth}
            \centering
            \includegraphics[width=\textwidth]{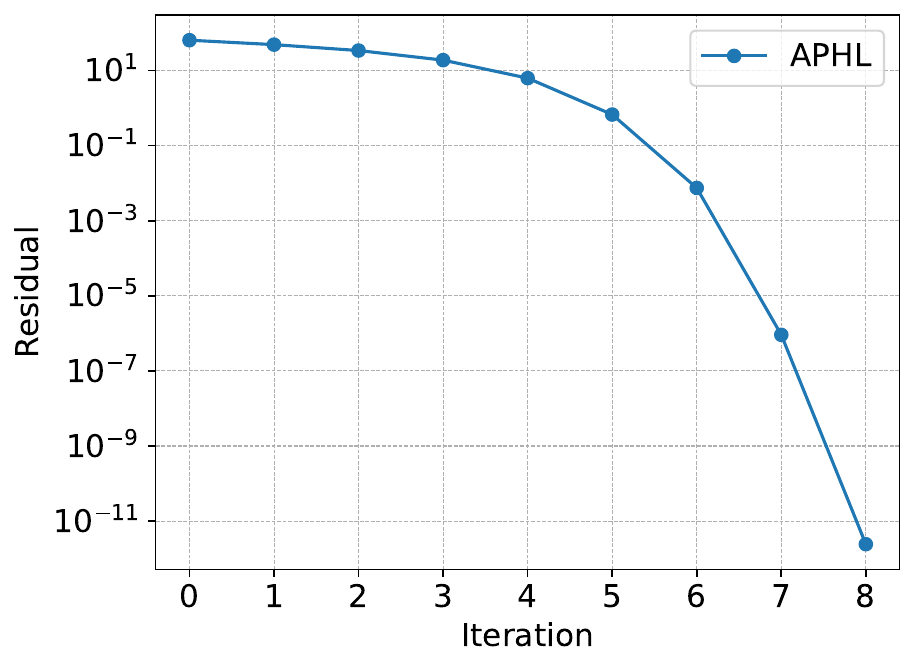}
        \end{minipage}
    }
    \subfigure[$$(n,m,p,r)=(1000,1000,50,100)$$]{
        \begin{minipage}[t]{0.30\textwidth}
            \centering
            \includegraphics[width=\textwidth]{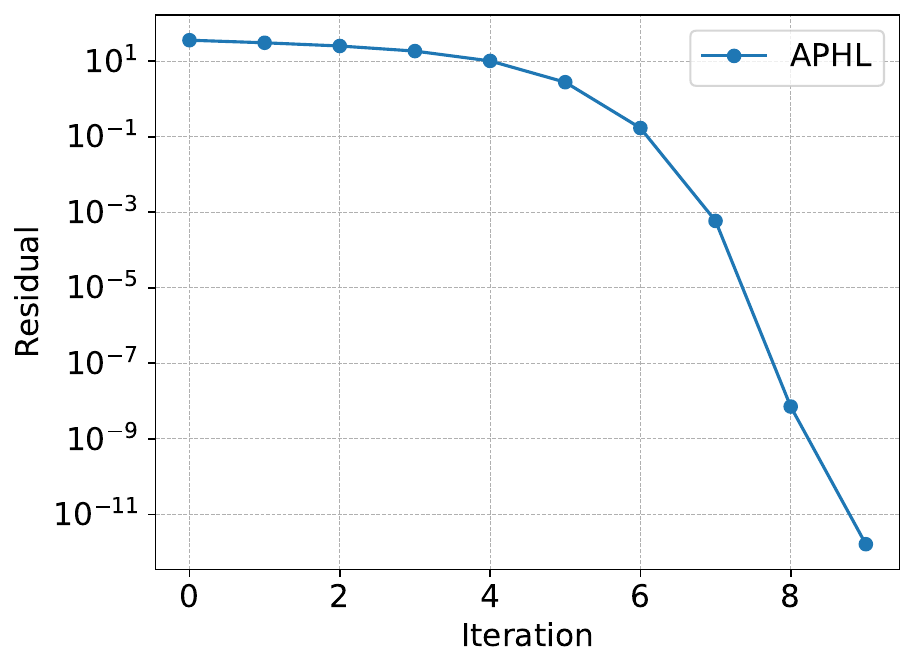}
        \end{minipage}
    }
    \subfigure[$$(n,m,p,r)=(1000,1000,500,100)$$]{
        \begin{minipage}[t]{0.30\textwidth}
            \centering
            \includegraphics[width=\textwidth]{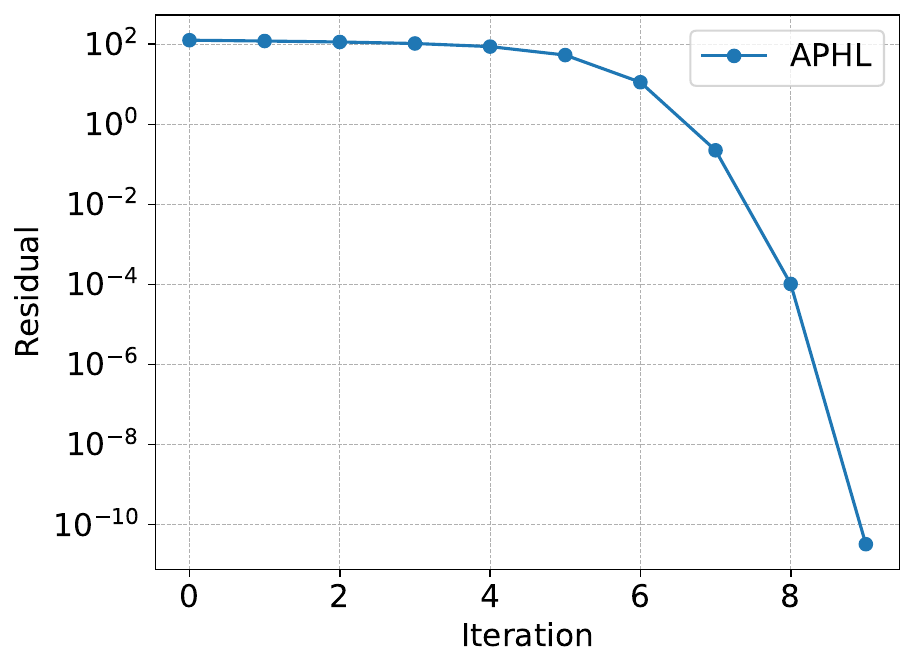}
        \end{minipage}
    }
    
\end{figure}

In this subsection, we consider the feasibility problem of the intersection of a low-rank variety and an affine space, which takes the following formulation:
\begin{equation}
    \label{Example_lr_variety}
    \mathrm{Find}\quad X \in  \M\cap{\cal D}_r,
\end{equation}
where ${\cal D}_r: \{X \in \bb{R}^{n\times m}: \mathrm{rank}(X) \leq r\}$. Moreover,  $\M=\{X\in\mathbb{R}^{n\times m}: \langle H_i,X\rangle=b_i, i \in [p]\}$ is an affine space in $\mathbb{R}^{n\times m}$, where 
$\{H_i: i\in [p]\}$ are matrices in $\bb{R}^{n \times m}$ and 
$b_i\in \bb{R}\;\forall\; i \in [p]$. 

We randomly generate the  matrices $\{H_i: i\in [p]\}$ by $H_i = \mathrm{numpy.random.randn}(n,m)$. 
To generate a feasible reference point $X_{\mathrm{ref}} \in \ca{D}_r$, we first randomly generate $W_{\rm{ref}} = \mathrm{numpy.random.randn}(n,m)$, and compute its singular value decomposition $W_{\rm{ref}} = U_{\rm{ref}} \Sigma_{\rm{ref}} V_{\rm{ref}}\tp$. Then the reference point $X_{\rm{ref}}$ is computed by $X_{\rm{ref}} = U_{\rm{ref}} \max\{\Sigma_{\rm{ref}} , 0\}V_{\rm{ref}}\tp$. Additionally, to enforce the nonempty-ness of $\M\cap{\cal D}_r$, we set $b_i = \inner{H_i, X_{\rm{ref}}}$ for $i \in [p]$. 

To better evaluate the numerical performance of Algorithm \ref{Alg:Adap_AP2}, we implement the NewtonSLRA solver \cite[Algorithm 2]{schost2016quadratically} for comparison. The maximum iteration of NewtonSLRA is set to be $300$, while all the other parameters are set as their default values as in \cite{schost2016quadratically}. Moreover, NewtonSLRA is set to stop when $\norm{\Pi_{{\cal D}_r}(X_k)}<10^{-10}$. 

It is worth noting that the initial point $X_0$ of NewtonSLRA needs to stay within $\M$ to guarantee that the whole sequence belongs to $\M$. Conversely, Algorithm \ref{Alg:Adap_AP2} requires the initial point $X_0 \in \D_r$. As a result, in the numerical experiments, we first randomly generate $\hat{X}_0$ by $\hat{X}_0 = \mathrm{numpy.random.randn}(n,m)$. Then we set $\Pi_{\M}(\hat{X}_0)$ as the initial point of NewtonSLRA, and set $\Pi_{\D_r}(\hat{X}_0)$ as the initial point of Algorithm \ref{Alg:Adap_AP2}.

The numerical results are presented in Table \ref{Table_num1_lr}, while the performance curves of Algorithm \ref{Alg:Adap_AP2} and NewtonSLRP are presented in Figure \ref{fig:lr}. It can be concluded from these results that although NewtonSLRA takes slightly fewer iterations than Algorithm \ref{Alg:Adap_AP2}, its per-iteration costs are significantly higher than those of Algorithm \ref{Alg:Adap_AP2}. 

Moreover, in the last six examples in Table \ref{Table_num1_lr} 
with large $n$ or $m$, NewtonSLRA could fail because of out of memory. More precisely, we employ the ``null space" function in NumPy to compute a basis of $\M$, and that function fails when $n$ or $m$ exceeds $1000$. Conversely, Algorithm \ref{Alg:Adap_AP2} demonstrates robust and efficient performance throughout all the test instances. These numerical results further demonstrate the advantages of Algorithm \ref{Alg:Adap_AP2} over the existing solver NewtonSLRA.

%%%%%%%%%%%%%%%%%%%%%%%%%%%%%%
\subsection{Quadratic programming with non-negative constraints}
In this subsection, we consider the following feasibility problem 
of the intersection of the nonnegative cone and a finite number of 
possibly nonconvex quadratic constraints:
\begin{equation}
    \label{Example_QP}
    \mathrm{Find} ~ x \in \bb{R}^n_+, \quad \inner{x, H_i x} = b_i, ~ \forall i \in [p].  
\end{equation}
Here $\{H_i: i \in [p]\}$ are symmetric matrices and $b_i\in \bb{R} \;\forall\; i \in [p]$. Then \eqref{Example_QP} fits into \eqref{Prob_NCF} with $\M=\{x\in\mathbb{R}^{n}\mid x^\top H_ix=b_i,\,i=1,\dots,p\}$ and $\X=\{x\in\mathbb{R}^{n}\mid x\geq0\}$. The generation of $X_{\rm{ref}}$, 
$\{H_i\}$ and $\{b_i\}$ are similar to those in Section \ref{Subsection_lr}. Furthermore, the initial point is chosen as $X_0 = X_{\rm{ref}} + 0.1 * \mathrm{random.randn}(n,1)$. Furthermore, as the projection onto $\M$ is generally intractable, we only test the numerical performance of Algorithm \ref{Alg:Adap_AP2}
but not the standard alternating projection  method. 

The numerical experiments are presented in Table \ref{Table_num_QP}, and the corresponding performance curves are presented in Figure \ref{fig:QP}. From these figures, we can observe that Algorithm \ref{Alg:Adap_AP2} is able to solve \eqref{Example_QP} to high precision with a small number of iterations. Moreover, Algorithm \ref{Alg:Adap_AP2} achieves quadratic convergence rate in solving \eqref{Example_QP}. 

\begin{table}[tb]
	\centering
	\caption{The numerical performance of Algorithm \ref{Alg:Adap_AP2} for solving \eqref{Example_QP}.}
	\label{Table_num_QP}
	\begin{tabular}{cc|ccc}
		\hline
		&           & iter & feas & time (s) \\ \hline
		\multicolumn{1}{l|}{\multirow{2}{*}{$n=100$}}& $p=10$   & 4  & 4.74e-14  & 0.01\\
		\multicolumn{1}{l|}{}                      & $p=50$      & 5  & 9.97e-14  & 0.13\\
        %\multicolumn{1}{l|}{} & $p=400$       & 9  & 4.04e+01  & 118.73\\
		\hline
        \multicolumn{1}{l|}{\multirow{3}{*}{$n=500$}}& $p=10$   & 3  & 7.60e-13  & 0.10\\
		\multicolumn{1}{l|}{}                      & $p=100$     & 4  & 2.38e-12  & 2.80\\
        \multicolumn{1}{l|}{}                      & $p=250$      & 5  & 3.37e-12  & 41.81 \\
		\hline
        \multicolumn{1}{l|}{\multirow{3}{*}{$n=1000$}}& $p=10$   & 3  & 2.41e-12  & 0.34\\
		\multicolumn{1}{l|}{}                      & $p=100$     & 4  & 9.65e-12  & 6.51\\
        \multicolumn{1}{l|}{}                      & $p=500$     & 5  & 1.65e-11  & 653.91\\
		\hline
	\end{tabular}
\end{table}

\begin{figure}[htbp]
    \centering
    % 第一行 Case 1
    \caption{The performance curves of Algorithm \ref{Alg:Adap_AP2} for solving \eqref{Example_QP}.}
    \label{fig:QP}
    \subfigure[$n=500, p=10$]{
        \begin{minipage}[t]{0.30\textwidth}
            \centering
            \includegraphics[width=\textwidth]{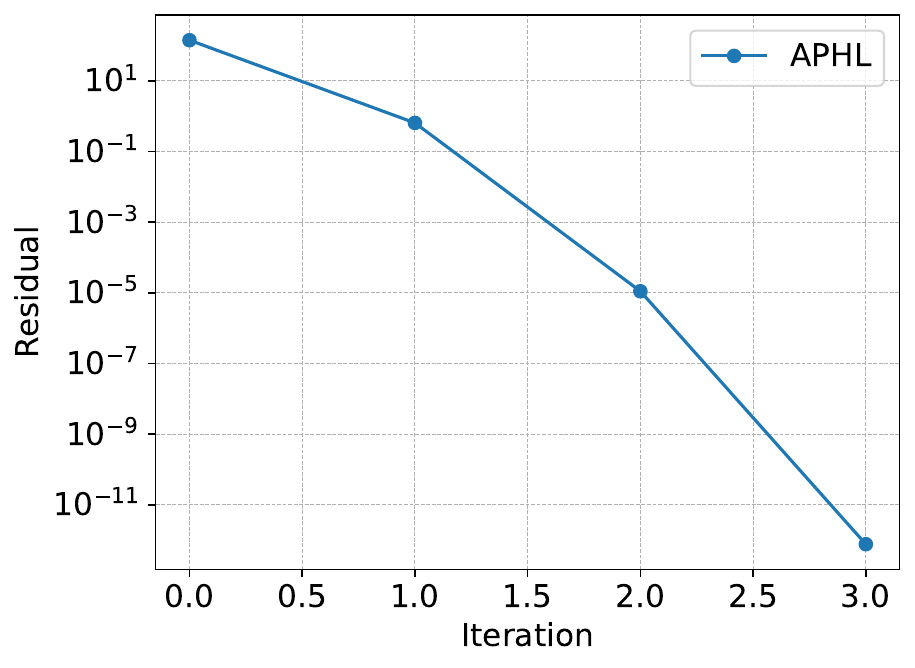}
        \end{minipage}
    }
    \subfigure[$n=500, p=100$]{
        \begin{minipage}[t]{0.30\textwidth}
            \centering
            \includegraphics[width=\textwidth]{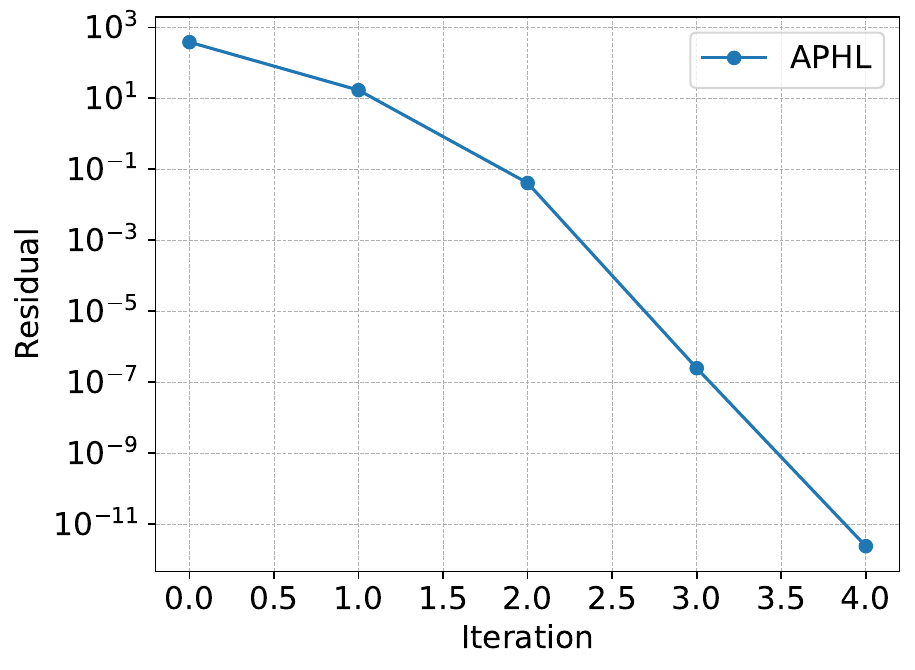}
        \end{minipage}
    }
    \subfigure[$n=500, p=250$]{
        \begin{minipage}[t]{0.30\textwidth}
            \centering
            \includegraphics[width=\textwidth]{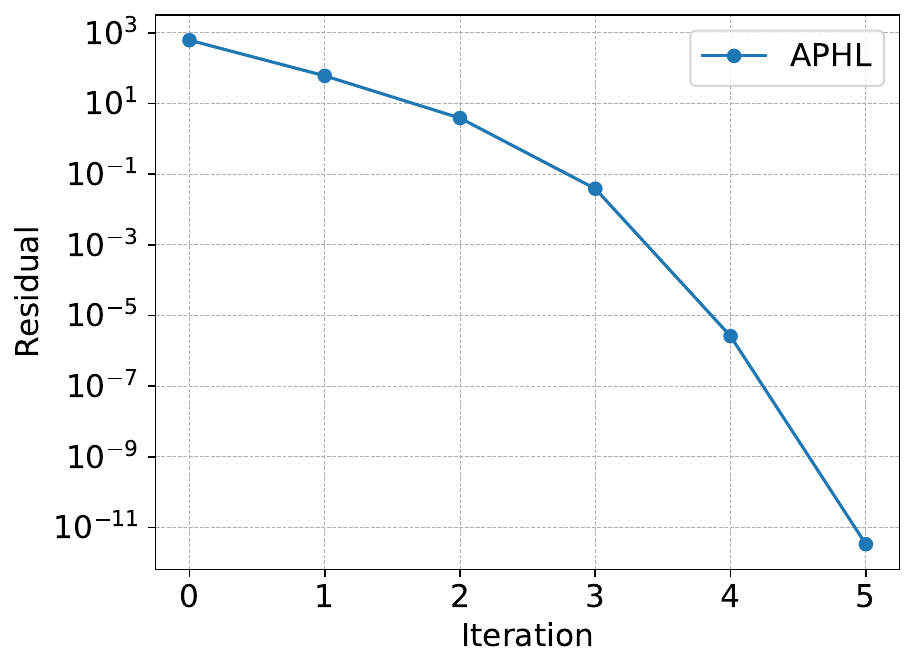}
        \end{minipage}
    }
    
    \subfigure[$n=1000, p=10$]{
        \begin{minipage}[t]{0.30\textwidth}
            \centering
            \includegraphics[width=\textwidth]{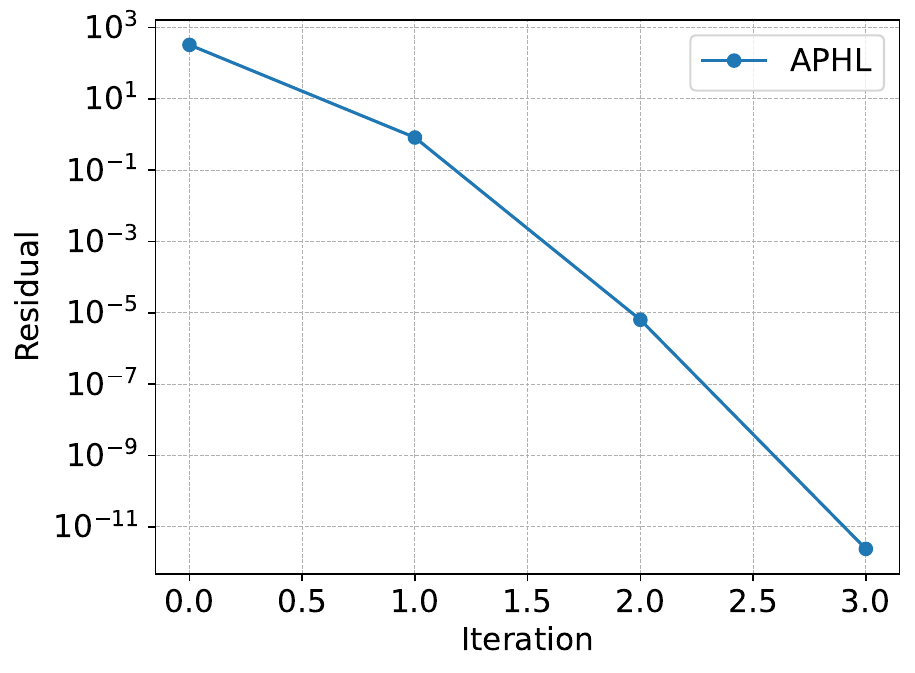}
        \end{minipage}
    }
    \subfigure[$n=1000, p=100$]{
        \begin{minipage}[t]{0.30\textwidth}
            \centering
            \includegraphics[width=\textwidth]{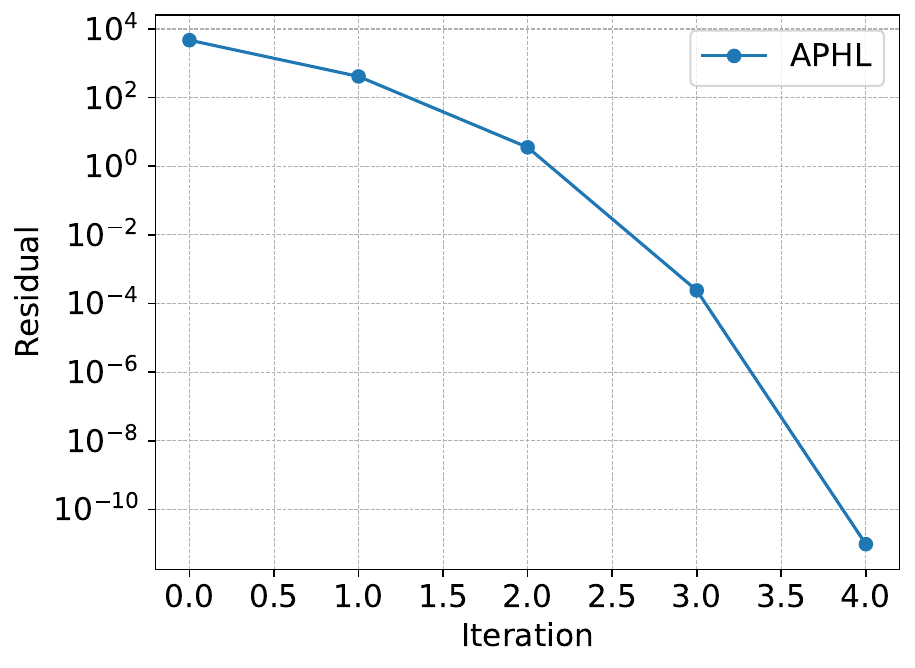}
        \end{minipage}
    }
    \subfigure[$n=1000, p=500$]{
        \begin{minipage}[t]{0.30\textwidth}
            \centering
            \includegraphics[width=\textwidth]{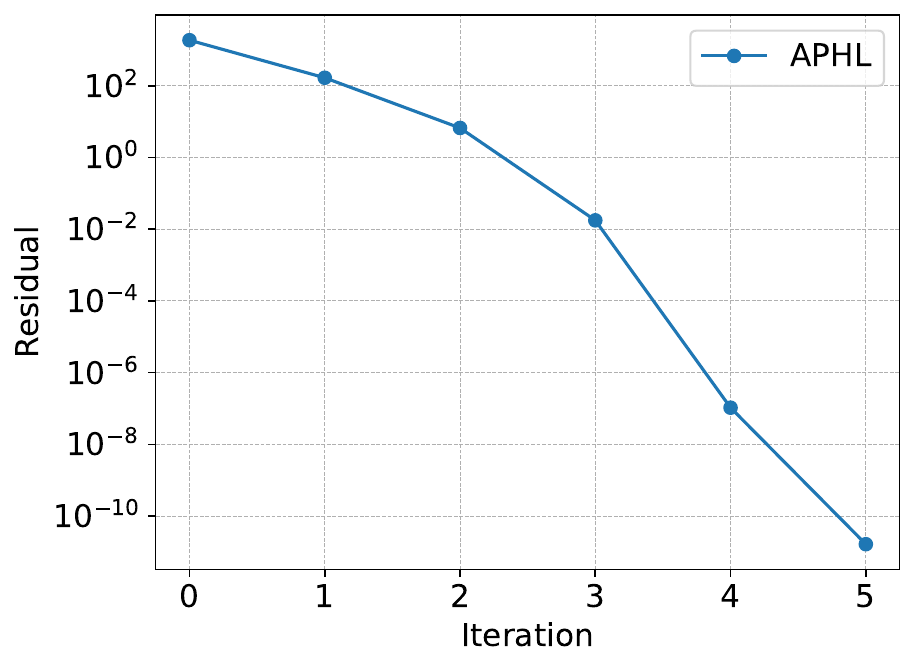}
        \end{minipage}
    }
    
\end{figure}

\subsection{$l_q$-norm ball with quadratic constraint}
 In this subsection, we consider the feasibility problem of the intersection of a $\ell_{1/2}$-norm ball and an affine space, which takes the following formulation:
 \begin{equation}
     \label{Example_ell_12}
     \mathrm{Find} ~ x \in \Rn, \text{s. t.} \quad \norm{x}_{1/2} \leq 1, ~ \langle x, H\tp x\rangle +B^T x = b. 
 \end{equation}
 Here $\norm{x}_{1/2} := (\sum_{i \in[n]} \sqrt{|x_i|})^2$ refers to the $\ell_{1/2}$-norm of a given vector, $H, B \in \bb{R}^{n \times p}$ and $b \in \Rp$. Then \eqref{Example_ell_12} fits into \eqref{Prob_NCF} with  $\M=\{x\in\Rn: \langle x, H\tp x\rangle +B^T x = b\}$ and
 $\X=\{x\in\mathbb{R}^n: \norm{x}_{1/2}\leq 1\}$, where $\X$ is non-regular \cite{clarke1990optimization}.  

 In our numerical experiments, we generate the reference point $x_{\rm{ref}} = \Pi_{\X}(\mathrm{random.randn}(n,1))$. Additionally, we generate  $H = \mathrm{random.randn}(n,p)$ and set $b = H\tp x_{\rm{ref}}$. Moreover, the initial point $x_0 = x_{\rm{ref}} + 10^{-5} * \mathrm{random.randn}(n,1)$. Additionally, we adopt the iterative reweighted projection method in \cite{yang2022towards} to compute $\Pi_\X$.
 
The numerical results are presented in Table \ref{Table_num_lq} and Figure \ref{fig:i12}. From these numerical results, we can observe that Algorithm \ref{Alg:Adap_AP2} successfully solves all the feasibility problems to high accuracy within a moderate number of iterations. Moreover, Algorithm \ref{Alg:Adap_AP2} demonstrates quadratic convergence rate in solving \eqref{Example_ell_12}, which coincides with the results presented in Theorem \ref{Theo_quadconv}. These results further demonstrate the high efficiency of Algorithm \ref{Alg:Adap_AP2} in solving \eqref{Prob_NCF} even
with a non-regular $\X$.

\begin{rmk}
    For a general constrained optimization problem, computing its first-order stationary points can be transformed into a feasibility problem based on its first-order optimality condition. However, as the relative interior of the set of first-order stationary points is tiny or of zero measure,  the nondegeneracy condition may fail for these feasibility problems,  and as a result, alternating projection methods are not efficient in directly solving these feasibility problems. How to improve the performance of alternating projection methods in solving these variants of feasibility problems is an interesting topic for future research. 
\end{rmk}

\begin{table}[tb]
	\centering
	\caption{The numerical performance of Algorithm \ref{Alg:Adap_AP2} for solving \eqref{Example_ell_12}.}
	\label{Table_num_lq}
	\begin{tabular}{cc|ccc}
		\hline
		&           & iter & feas & time (s) \\ \hline
		\multicolumn{1}{l|}{\multirow{2}{*}{$n=100$}}& $p=10$   & 2  & 3.93e-14  & 0.01\\
		\multicolumn{1}{l|}{}                      & $p=50$      & 6  & 3.59e-16  & 0.10\\
		\hline
        \multicolumn{1}{l|}{\multirow{3}{*}{$n=500$}}& $p=10$   & 61  & 4.37e-12  & 3.40\\
		\multicolumn{1}{l|}{}                      & $p=100$     & 247  & 1.79e-14  & 44.24\\
         \multicolumn{1}{l|}{}  & $p=200$      & 885  & 3.80e-13  & 380.96\\
		\hline

        \multicolumn{1}{l|}{\multirow{3}{*}{$n=1000$}}& $p=10$   & 43  & 1.15e-11  & 6.69\\
		\multicolumn{1}{l|}{}                      & $p=100$     & 319  & 4.87e-12  & 196.16\\
        \multicolumn{1}{l|}{}                      & $p=200$     & 706  & 2.41e-13  & 1280.19\\
		\hline
	\end{tabular}
\end{table}

\begin{figure}[htbp]
    \centering
    % 第一行 Case 1
    \caption{The performance curves of Algorithm \ref{Alg:Adap_AP2} for solving \eqref{Example_ell_12}.}
    \label{fig:i12}
    \subfigure[$n=500, p=10$]{
        \begin{minipage}[t]{0.30\textwidth}
            \centering
            \includegraphics[width=\textwidth]{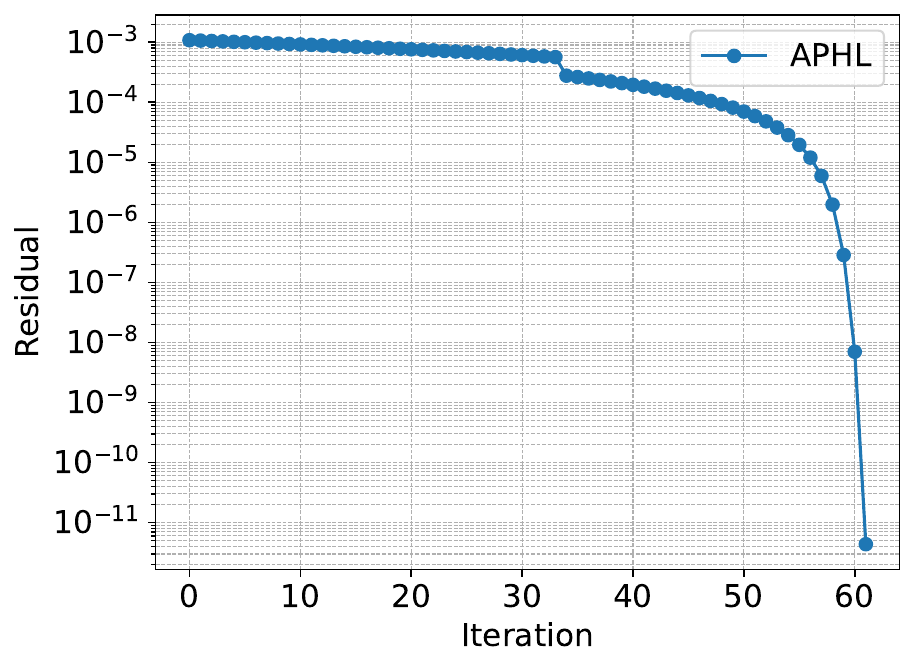}
        \end{minipage}
    }
    \subfigure[$n=500, p=100$]{
        \begin{minipage}[t]{0.30\textwidth}
            \centering
            \includegraphics[width=\textwidth]{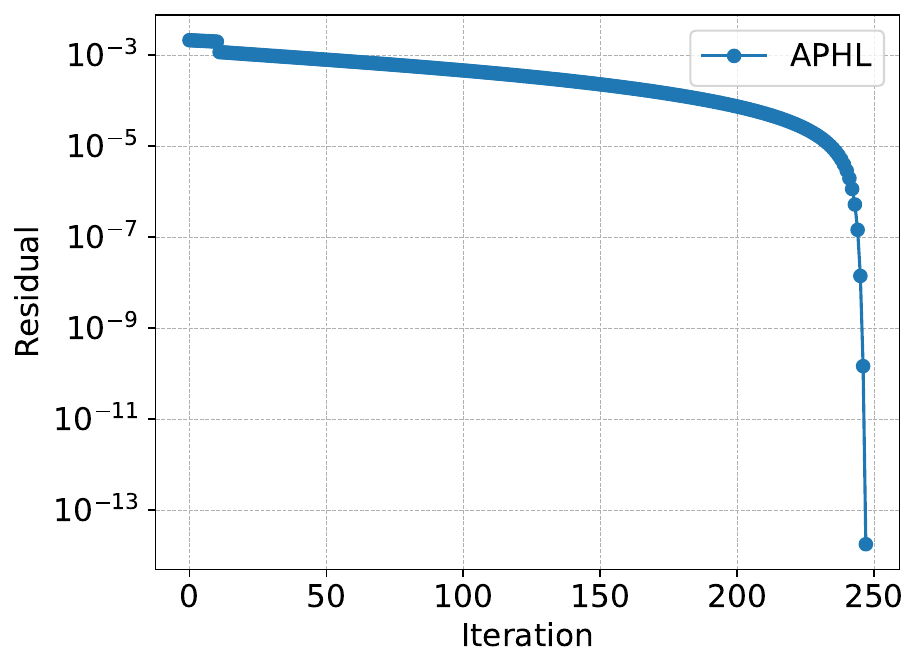}
        \end{minipage}
    }
    \subfigure[$n=500, p=200$]{
        \begin{minipage}[t]{0.30\textwidth}
            \centering
            \includegraphics[width=\textwidth]{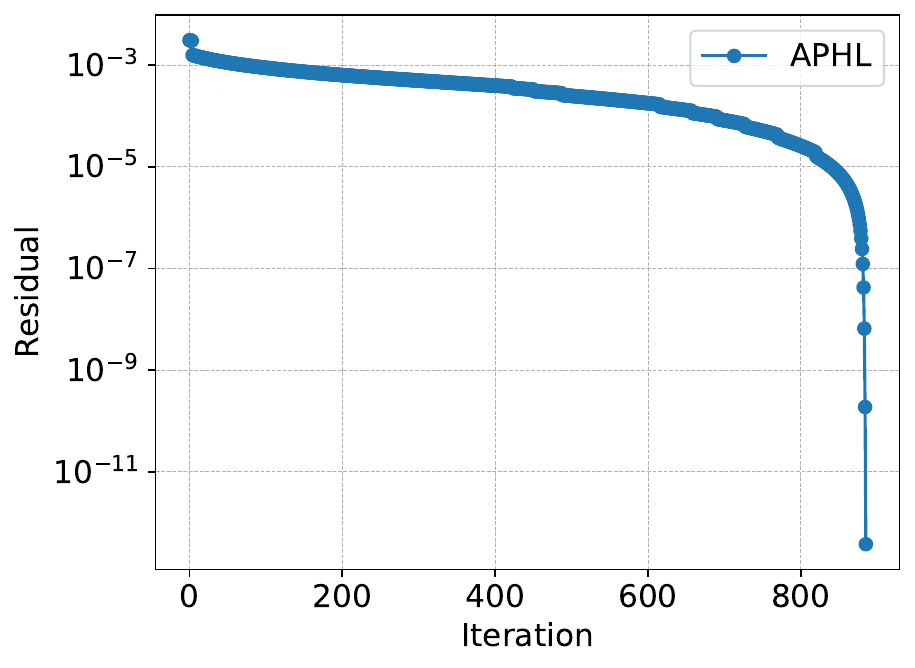}
        \end{minipage}
    }
    
    \subfigure[$n=1000, p=10$]{
        \begin{minipage}[t]{0.30\textwidth}
            \centering
            \includegraphics[width=\textwidth]{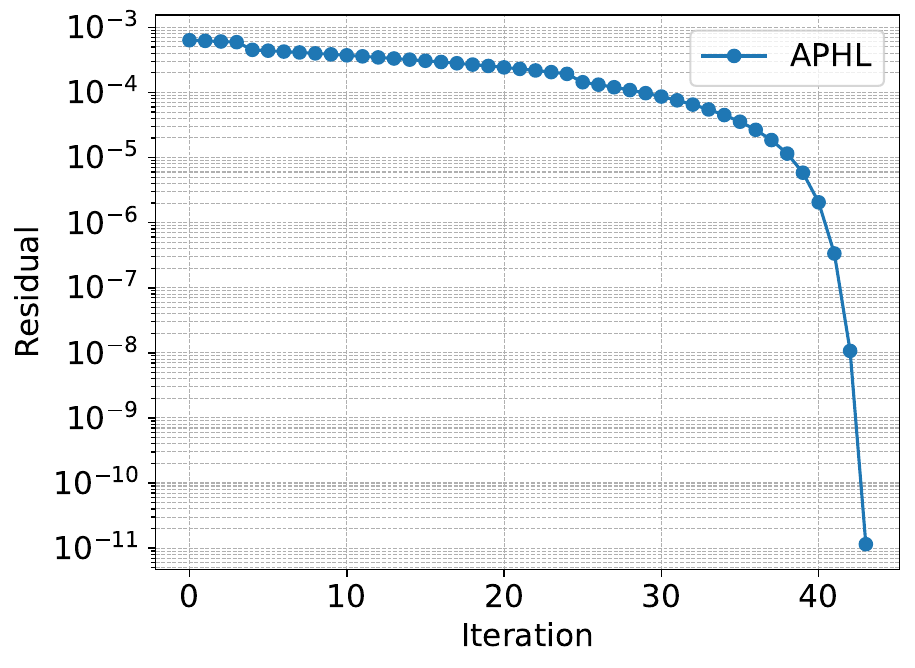}
        \end{minipage}
    }
    \subfigure[$n=1000, p=100$]{
        \begin{minipage}[t]{0.30\textwidth}
            \centering
            \includegraphics[width=\textwidth]{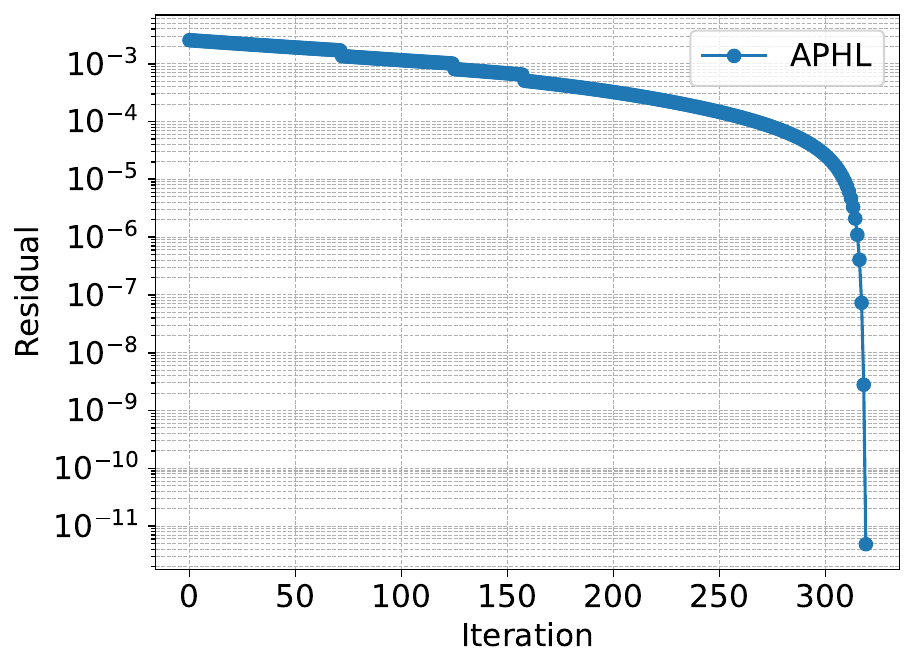}
        \end{minipage}
    }
    \subfigure[$n=1000, p=200$]{
        \begin{minipage}[t]{0.30\textwidth}
            \centering
            \includegraphics[width=\textwidth]{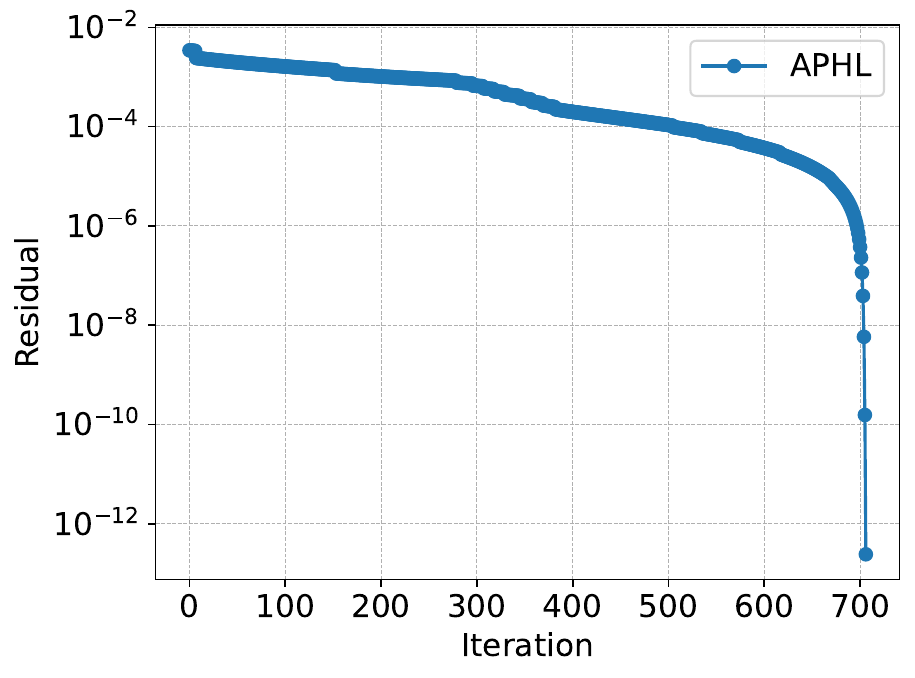}
        \end{minipage}
    }

\end{figure}

\section{Conclusion}
In this paper, we consider the nonconvex feasibility problem in the form of \eqref{Prob_NCF}.  The classical alternating projection method for solving \eqref{Prob_NCF} only has the local linear convergence rate, while the local quadratically convergent variant \cite{schost2016quadratically} is only designed for a very specific nonconvex feasibility problem (i.e., the intersection of a smooth manifold and an affine space). 

We propose a modified alternating projection scheme \eqref{Eq_Feasibility_RestorationQ} for solving the nonconvex feasibility problem \eqref{Prob_NCF}. Our proposed scheme \eqref{Eq_Feasibility_RestorationQ} admits a closed-form update scheme, which only requires the computations of matrix-matrix multiplications and solving a $p\times p$ linear system of equations in each iteration. We prove the local quadratic convergence rate for our proposed scheme \eqref{Eq_Feasibility_RestorationQ}. Moreover, by combining the scheme \eqref{Eq_Feasibility_RestorationQ} with the projected gradient method for solving \eqref{Prob_Con}, we propose a globally convergent algorithm in Algorithm \ref{Alg:Adap_AP2}. Preliminary numerical experiments illustrate the high efficiency of our proposed Algorithm \ref{Alg:Adap_AP2}.

\section*{Acknowledgments}
The authors express their gratitude to Dr.~Kuangyu Ding and Dr.~Di Hou for their valuable comments on alternating projection methods. 
 
\bibliographystyle{plain}
\bibliography{ref}

\end{document}